\documentclass[10pt]{amsart}
\newtheorem{theorem}{Theorem}[section]
\newtheorem{proposition}[theorem]{Proposition}

\newtheorem{corollary}[theorem]{Corollary}

\newtheorem{lemma}[theorem]{Lemma}
\newtheorem{example}[theorem]{Example}

\newtheorem{remark}[theorem]{Remark}

\newtheorem{definition}{Definition}[section]
\usepackage{amsmath,amssymb,amsfonts,color}
\usepackage{enumerate}
\usepackage{hyperref}
\newcommand{\Ga}{\Gamma}

\newcommand{\Cay}{\textrm{Cay}}
\newcommand{\Sym}{\textrm{Sym}}

\def\Sym{{\rm Sym}}
\def\ov{\overline}
\def\Cay{{\rm Cay}}

\def\la{\langle}
\def\ra{\rangle}
\def\Aut{{\rm Aut}}

\def\calA{\mathcal{A}}

\def\caycs{{\sf NormCay{CS}}}

\def\Ga{\Gamma}

\def\vp{\varphi}
\def\GL{{\rm GL}}
\def\ra{\rangle}

\usepackage{anysize}

\begin{document}

\title[Normal  edge-transitive Cayley graphs]
{Normal edge-transitive Cayley graphs and\\  Frattini-like subgroups}

\author[Behnam Khosravi \ \and Cheryl E. Praeger]
{$^1$ Behnam Khosravi and  $^2$ Cheryl E. Praeger}
\address{$^1$   Department of Mathematics,
	Institute for Advanced Studies in Basic Sciences (IASBS), Zanjan
	45137-66731, Iran.  \\ email \href{mailto:behnam\_kho@yahoo.com}{behnam\_kho@yahoo.com} \newline $^2$ Department of Mathematics and Statistics, The University of Western Australia, 35 Stirling Highway, Perth, WA 6009, Australia,
	email \href{mailto:cheryl.praeger@uwa.edu.au}{cheryl.praeger@uwa.edu.au}.}



\thanks{Dedicated to the memory of our colleague and friend Jan Saxl}

\thanks{The second author would like to thank the Isaac Newton Institute for Mathematical Sciences, Cambridge, for support and hospitality during the programme \emph{Groups, representations and applications: new perspectives} (supported by EPSRC grant no. EP/R014604/1), where work on this paper was undertaken.
}

\begin{abstract}
	For a finite group $G$ and an inverse-closed generating set $C$ of $G$, let
	$\Aut(G;C)$ consist of those automorphisms of $G$ which leave
	$C$ invariant.
	We define an $\Aut(G;C)$-invariant normal subgroup $\Phi(G;C)$ of $G$ which has the
	 property  that, for any
	$\Aut(G;C)$-invariant normal set of generators for $G$, if we remove from it all the elements of
	$\Phi(G;C)$, then the remaining set is still an $\Aut(G;C)$-invariant normal
	generating  set  for $G$. The subgroup $\Phi(G;C)$ contains the Frattini subgroup
	$\Phi(G)$
	but the inclusion may be proper.
	The Cayley graph $\Cay(G,C)$ is normal edge-transitive if $\Aut(G;C)$
	acts transitively on the
	pairs $\{c,c^{-1}\}$ from $C$.
	We show that, for a normal edge-transitive   Cayley
	graph $\Cay(G,C)$, its quotient modulo $\Phi(G;C)$ is the unique largest normal
	quotient which is isomorphic to a subdirect product of normal edge-transitive
	graphs of characteristically simple groups.  In particular, we may therefore view normal
	edge-transitive Cayley graphs of characteristically simple groups as building blocks
	for normal edge-transitive Cayley graphs whenever we have $\Phi(G;C)$
	 trivial. We explore several questions which these results raise, some concerned with the set of all inverse-closed generating sets for groups in a given family. In particular we use this theory to classify all $4$-valent normal edge-transitive Cayley graphs for dihedral groups; this involves a new construction of an infinite family of examples, and disproves a conjecture of Talebi.
\end{abstract}

\maketitle
\section{Introduction and the main results}

Let $G$ be a nontrivial finite group (that is, $|G|>1$), and let $\Aut(G)$ denote its automorphism group. A non-empty  subset $C$ of  $G\setminus\{1\}$ is said to be \emph{inverse-closed} if for each $c\in C$, its inverse $c^{-1}$ also lies in $C$.
We will be concerned with inverse-closed generating sets $C$ for $G$, and for such a set we let $\Aut(G;C)$ denote the subgroup of $\Aut(G)$ consisting of those automorphisms
$\sigma$  which leave	$C$ invariant, that is to say, $C^\sigma=C$. A subgroup $N\leq G$ is $\Aut(G;C)$-invariant if $N^\sigma=N$ for each $\sigma\in\Aut(G;C)$. We will study
\begin{equation}\label{eq:agc}
\calA(G;C) =\{ N\mid N\unlhd G, \ N\ne G, \ \mbox{and $N$ is  $\Aut(G;C)$-invariant}\}
\end{equation}
and in particular its subset of maximal elements
\begin{equation}\label{eq:amax}
\calA_{\max}(G;C) =\{ N\in\calA(G;C) \mid \ \mbox{$N$ is maximal by inclusion}\}.
\end{equation}
Since $|G|>1$, the identity subgroup lies in $\calA(G;C)$, and it follows that $\calA_{\max}(G;C)$ is non-empty. Thus the following subgroup
\begin{equation}\label{eq:phi}
\Phi(G;C) = \bigcap\{N \mid N\in\calA_{\max}(G;C)\}
\end{equation}
is a well defined element of $\calA(G;C)$.  We may have $\Phi(G;C)=1$; this is the case, for example, if $G$ is simple so that $\calA_{\max}(G;C)=\calA(G;C)=\{\{1\}\}$.
The definition of $\Phi(G;C)$ as the intersection of a certain family of subgroups of $G$ is reminiscent of the definition of the Frattini subgroup $\Phi(G)$ as the intersection of all maximal (proper) subgroups of $G$. Indeed $\Phi(G;C)$ contains $\Phi(G)$
(see Lemma~\ref{lem:frat}), but the inclusion can be proper (see Examples~\ref{ex:frat} and~\ref{ex:frat2}).

An important property of the Frattini subgroup $\Phi(G)$ is that, if $X$ is a generating set for $G$ and $X$ contains an element $y\in\Phi(G)$, then $X\setminus\{y\}$ still generates $G$ (see, for example, \cite[Satz III.3.2]{Huppert}). It turns out that $\Phi(G;C)$ has a similar property. We say that a subset $X$ of $G$ is \emph{normal $C$-closed} if for each $x\in X, g\in G, \sigma\in\Aut(G; C)$, the image $(x^g)^\sigma$ also lies in $X$. Clearly an arbitrary subset $X\subseteq G$ may be enlarged to a normal $C$-closed subset by adding to it each of the elements $(x^g)^\sigma$; the set thus obtained is called the \emph{normal $C$-closure} of $X$. We show that from each normal $C$-closed generating set we may remove all elements of $\Phi(G; C)$ (if any) and maintain these properties.

\begin{theorem}\label{thm:frat}
Let $G$ be a finite group and $C$ a generating set for $G$.
\begin{enumerate}
\item[(a)] If $X$ is a normal $C$-closed generating set for $G$, then $X\setminus (X\cap \Phi(G;C))$ is also a normal $C$-closed generating set.
\item[(b)] The subgroup $\Phi(G;C)$ is the set of all elements $y\in G$ which satisfy the following property:

\begin{center}
for $X\subseteq G$, if the normal $C$-closure of $X\cup\{y\}$ generates $G$,\\
then the normal $C$-closure of $X$ also generates $G$.
\end{center}
\end{enumerate}
\end{theorem}

We will show that this Frattini-like subgroup $\Phi(G;C)$ plays a key role in describing the structure of normal edge-transitive Cayley graphs.

\emph{We briefly summarise the key concepts:}
for a group $G$, and a non-empty, inverse-closed subset $C\subseteq G\setminus\{1\}$, the
(undirected)
 {\it Cayley graph }$\Cay(G; C)$ {\it of $G$ with respect  to $C$} is the graph with
vertex set $G$ and edge set $E(G;C)$ consisting of all
pairs $\{g, h\}$ such that  $gh^{-1}\in C$ (or equivalently $hg^{-1}\in C$ since $C$ is inverse-closed). The graph $\Gamma:=\Cay(G; C)$ is connected if and only if the {\it connection set} $C$ generates $G$. For each
$g\in G$, the right multiplication map $\rho_g:h\mapsto  hg$ is an automorphism of
$\Gamma$, and the group $G_R:=\{\rho_g\mid g\in G\}$ (called {\it the right regular representation of $G$}) is a subgroup of $\Aut(\Gamma)$ acting regularly on vertices (that is, $G_R$ is transitive with trivial stabilisers).  In addition $\Aut(G;C)$ is a subgroup of $\Aut(\Gamma)$, and the semidirect product $N(G;C):= G_R\rtimes\Aut(G;C)$ is the full normaliser of $G_R$ in $\Aut(\Gamma)$ (see \cite{Godsil-auto}). As in
 \cite{praeger}, $\Gamma$ is said to be  {\it normal edge-transitive} if $N(G;C)$
is transitive on the edge-set $E(G;C)$. In \cite[Theorem 3]{praeger} it was shown that, for each $N\in\calA(G;C)$, the normal quotient, modulo $N$, of a connected normal edge-transitive Cayley graph $\Cay(G; C)$ is again a connected normal edge-transitive Cayley graph, namely
$\Cay(G/N; CN/N)$. In particular if $N\in\calA_{\max}(G;C)$, then the group $G/N$ is characteristically simple, and we refer to $\Cay(G/N; CN/N)$ as a \emph{characteristically simple normal quotient} of $\Cay(G;C)$.

We denote  the set of connected  normal edge-transitive Cayley graphs for finite characteristically simple groups by $\caycs$.
A (still open) question from \cite[Question 2]{praeger} asks
under what conditions $\Cay(G; C)$ is determined by its proper normal quotients (that is, quotients modulo subgroups  $N\in\calA_{\max}(G;C)$ with $N\ne 1$). Our results suggest that, provided  $\Phi(G;C)=1$, $\Cay(G; C)$ is determined to a large extent by its normal quotients lying in $\caycs$.
Our main result relies on certain graph constructions which we define in Subsection~\ref{sub:graphnot}, namely the direct product $\prod_{i=1}^n\Ga_i$ of a family of graphs $\Ga_i$, and a full subdirect product of a graph direct product. A different graph theoretical construction was used recently in \cite{KKK} to study normal edge-transitive Cayley graphs $\Cay(G;C)$ of groups $G$ admitting a direct decomposition $G=M\times N$ with both $M, N\in  \calA(G;C)$, but the analysis in \cite{KKK} does not apply in our general situation.

\begin{theorem}\label{main decomposition}
	Let $G$ be a nontrivial finite group with inverse-closed generating set  $C$ such that $\Ga=\Cay(G,C)$ is normal edge-transitive, and let $\Phi:=\Phi(G;C)$ (as in $\eqref{eq:phi}$). Then
	there exist $k\geq 1$ and $N_1,\dots,N_k\in\calA_{\max}(G;C)$ for which the following conditions hold, where $\Ga_i:= \Cay(G/N_i;CN_i/N_i)$.
	\begin{enumerate}[{\rm (i)}]
		\item $\Phi=\cap_{i=1}^kN_i$, and the map $\zeta: g\Phi\rightarrow (gN_1,\cdots,gN_k)$ is an isomorphism $G/\Phi \rightarrow \prod_{i=1}^k G/N_i$;
		\item  for each $i\leq k$, $\Ga_i\in \caycs$, $\calA(G/N_i; CN_i/N_i)=\{1\}$, and $\zeta$ induces a graph isomorphism from $\Cay(G/\Phi; C\Phi/\Phi)$ onto a full subdirect product of $\prod_{i=1}^k\Ga_i$.
	\end{enumerate}
\end{theorem}

\begin{remark}{\rm
Theorem~$\ref{main decomposition}$ constructs a normal quotient  $\Cay(G/\Phi(G;C); C\Phi(G;C)/\Phi(G;C))$ of a given normal edge-transitive Cayley graph $\Cay(G;C)$
which is a full subdirect product of a family of its normal quotients lying in $\caycs$. If $\Phi(G;C)=1$ this gives a structural description of the original graph $\Cay(G;C)$.
}
\end{remark}

Theorem~\ref{main decomposition} follows fairly easily from a more general result Theorem~\ref{main} which considers general normal quotients in $\caycs$.
We prove Theorem~\ref{thm:frat} in Section~\ref{sec:frat}, and Theorem~\ref{main decomposition} in Section~\ref{sec:netg}.
In Section~\ref{sec:dehedral}, we work through this theory for dihedral groups, examining in Theorem~\ref{lem:d2n} the possible structures of inverse-closed generating sets for such groups which are transitive in the sense of Definition~\ref{def:trans}. By \cite[Proposition 1(b)]{praeger}, these are precisely the connection sets which lead to normal edge-transitive Cayley graphs. In particular we classify all $4$-valent normal edge-transitive Cayley graphs for dihedral groups.

\begin{theorem}\label{4valent}
Let $G=D_{2n}$ with $n\geq3$, and let $C$ be an inverse-closed generating set for $G$ with $|C|=4$. Then  $\Cay(G;C)$ is a $4$-valent normal edge-transitive Cayley graph if and only if, replacing $C$ by $C^\sigma$ for some $\sigma\in\Aut(G)$ if necessary, $C$ is as in Example~$\ref{ex:talebi}$ or Example~$\ref{ex:new}$.
\end{theorem}

The graphs corresponding to subsets $C$ in Example~$\ref{ex:talebi}$ were constructed by Talebi~\cite{Talebi}, and were conjectured to be the only examples. Our construction in Example~$\ref{ex:new}$ contains infinitely many new examples, as demonstrated in Subsection~\ref{sub-bk}.

We note that  this and previous work suggests some open questions, and we mention a few in the next remark.

\begin{remark}{\rm
(a)\ We will see in Section~\ref{sec:frat} that the `relative Frattini subgroups' $\Phi(G;C)$ may be equal to, or may be much larger than the Frattini subgroup. They can be soluble or insoluble, even for transitive inverse-closed generating sets $C$. It is not clear what general properties they may have. Any additional information would help in understanding the structure of the corresponding normal edge-transitive Cayley graphs $\Cay(G;C)$.

\smallskip
(b)\ Let $C$ be a transitive inverse-closed generating set for a group $G$  in the sense of Definition~\ref{def:trans}, and suppose that the corresponding normal edge-transitive Cayley graph $\Cay(G;C)$ has valency $k$. It follows from Lemma~\ref{lem:notinphi} that the  normal edge-transitive quotient $\Cay(\ov{G};\ov{C})$, where $\ov{G}=G/\Phi(G;C)$ and $\ov{C}=C\Phi(G;C)/\Phi(G;C)$, has valency $k/\ell$, with $\ell$ as in
Lemma~\ref{lem:notinphi}(b).  It is not quite clear what values the parameter $\ell$ can take. Understanding this better would help in understanding better the structure of $\Cay(G;C)$.

\smallskip
(c)\ In Theorem~\ref{main decomposition}, we make no claims about the uniqueness of the normal quotient graphs $\{\Gamma_1,\dots,\Gamma_k\}$ in $\caycs$, or the subgroups $N_1,\dots,N_k\in\calA_{\max}(G;C)$. It would be good to understand this better. In particular, are there examples where these subgroups and graphs are not uniquely determined by $G, C$?

\smallskip
(d)\  For dihedral groups $G$, Theorem~\ref{lem:d2n} describes the family $\mathcal{T}(G)$  of transitive inverse-closed generating sets $C$, and the corresponding subgroups $\Aut(G;C)$ of $\Aut(G)$.  It would be interesting to explore the set $\mathcal{T}(G)$ for other families of groups, to shed light on the structure of normal edge-transitive Cayley graphs for these groups.
 The sets  $\mathcal{T}(G)$ for $G$ a Frobenius group  of order $qp$ (with $p$ a prime) may be extracted from the work of Darafsheh and Assari~\cite[Section 3]{Dar Ass} if $q=4$, and   work of Corr and the second author \cite{Corr Pra} if $q$ is a prime; see also \cite{ATA} if $q=3$. Those for nonabelian groups of order $4p^2$ ($p$ a prime) are determined in \cite{PI}. Similarly cyclic groups of prime power oprder are treated in \cite{SK}, abelian groups of order a power of two primes are treated in  \cite{Corr Pra}, generalised quaternion groups of order $4p$ (with $p$ a prime) in~\cite[Section 3]{Dar Ass}, and the $4$-element subsets for certain groups of order $6n$ in \cite{Dar Yag}. It would be good to see more general treatments.
}
\end{remark}



\section{The Frattini-like subgroup $\Phi(G;C)$}\label{sec:frat}

We first prove the assertion that $\Phi(G;C)$ contains the Frattini subgroup.

\begin{lemma}\label{lem:frat}
Let $G$ be a nontrivial finite group with an inverse-closed generating set $C\subseteq G\setminus\{1\}$. Let $\Phi(G;C)$ be as in $\eqref{eq:phi}$, and let $\Phi(G)$ be the Frattini subgroup. Then
\begin{enumerate}
\item[(a)] $\Phi(G)\leq \Phi(G;C)$;
\item[(b)] for $N\in\calA_{\max}(G;C)$ and $N\leq M<G$ with $M$ maximal in $G$, we have
$N=\cap_{g\in G,\sigma\in\Aut(G; C)}(M^g)^\sigma$.
\end{enumerate}
\end{lemma}

\begin{proof}
Let $\mathcal{M}$ be the set of maximal proper subgroups of $G$, so by definition, $\Phi(G)=\cap_{M\in\mathcal{M}}M$. For each $M\in \mathcal{M}$, and each $g\in G, \sigma\in\Aut(G;C)$, the image $(M^g)^\sigma$ is also a subgroup in $\mathcal{M}$, and so $\Phi(G)\leq N_M$,
where $N_M:=\cap_{g\in G,\sigma\in\Aut(G; C)}(M^g)^\sigma$. Note that, by definition, $N_M\in
\calA(G; C)$.

Now let $N\in \calA_{\max}(G;C)$, and let $M\in \mathcal{M}$ such that $N\leq M$ (note that such an $M$ exists since $G$ is finite).  Then $N\leq N_M$ since $N$ is normal in $G$ and $\Aut(G; C)$-invariant. By the maximality of $N$, it now follows that $N=N_M$, proving part (b).
 Since this holds for all $N\in \calA_{\max}(G;C)$, it follows that  $\Phi(G)\leq \cap_{N\in \calA_{\max}(G;C)} N = \Phi(G;C)$, proving part (a).
\end{proof}

Sometimes equality holds in Lemma~\ref{lem:frat}(a), for example when $G$ is a cyclic group (Example~\ref{ex:frat0}).  It is also easy to construct examples for which equality does not hold, for example, when $G$ is a dihedral group (Example~\ref{ex:frat}), or is almost simple but not simple (Example~\ref{ex:frat2}). Example~\ref{ex:frat} shows how $\Phi(G;C)$ can vary for different inverse-closed generating sets $C$.

\begin{example}\label{ex:frat0}
{\rm
Let $G=\la a\mid a^n=1\ra\cong Z_{n}$ with $n\geq2$. Each subgroup of $G$ is normal and $\Aut(G)$-invariant. In particular the maximal subgroups of $G$ are the cyclic subgroups $M_p=\la a^p\ra\cong Z_{n/p}$, for  primes $p$ dividing $n$. These are also the maximal normal $\Aut(G;C)$-invariant subgroups for each inverse-closed generating set $C$ of $G$. Thus
 the Frattini subgroup (which is the intersection of all the maximal subgroups of $G$), and also  $\Phi(G;C)$, are equal to
\[
\Phi(G)=\Phi(G;C)=\la a^m\ra,\ \mbox{where $m$ is the product of the distinct prime divisors of $n$.}
\]
In particular, $\Phi(G)=\Phi(G;C)=1$ if and only if $n$ is square-free.
}
\end{example}

The next example is for the dihedral group
\begin{equation}\label{eqd}
G=\la a, b\mid a^n=b^2=1, bab=a^{-1}\ra \cong D_{2n}
\end{equation}
of order $2n$. This group $G$ admits automorphisms $\vp, \tau_k$, for positive integers $k<n$ such that $\gcd(k,n)=1$, defined on the generators $a, b$ as follows:
\begin{equation}\label{eqvpt}
\vp: a\to a,\ b\to ba; \quad \tau_k: a\to a^k,\ b\to b,
\end{equation}
and we have
\begin{equation}\label{eqautd}
\Aut(G) = A\rtimes T,\ \mbox{where}\ A=\la\vp\ra\cong Z_n,\ \mbox{and}\ T=\{ \tau_k \mid 1\leq k\leq n-1,\ \gcd(k,n)=1\} \cong \Aut(Z_n).
\end{equation}

\begin{example}\label{ex:frat}
{\rm
Let $G\cong D_{2n}$, as in \eqref{eqd}, with $n\geq3$. If $n$ is odd then the $n$ involutions in $G$ form a single $G$-conjugacy class. On the other hand, if $n$ is even then the $n$ non-central involutions form two $G$-conjugacy classes of size $n/2$, with representatives $b, ba$, and $\vp\in\Aut(G)$ interchanges these two classes.

The maximal subgroups of $G$ are of two types. Firstly there is the cyclic subgroup $M=\la a\ra\cong Z_n$. Secondly, for each prime $p$ dividing $n$, we have a dihedral subgroup
\[
M_p:= \la a^p,b\ra\cong D_{2n/p} \ \mbox{and also, if $n$ is even and $p=2$,} \ M_2':=\la a^2, ba\ra\cong D_n.
\]
If $n$ is even then the two subgroups $M_2, M_2'$ are distinct index $2$ (hence normal) subgroups of $G$, and $M_2\cap M_2'=\la a^2\ra$. Thus, for arbitrary $n$,
 the Frattini subgroup (which is the intersection of all the maximal subgroups of $G$) is equal to
\[
\Phi(G)=\la a^m\ra,\ \mbox{where $m$ is the product of the distinct prime divisors of $n$.}
\]
For example, $\Phi(G)=1$ if and only if $n$ is square-free.
Now we consider $\Phi(G;C)$ for two particular inverse-closed generating sets $C$ of $G$, namely
\[
C_1=\{a, a^{-1},b\},\quad \mbox{and}\quad C_2=\{b, ba\}.
\]
In each case, to determine $\Phi(G;C)$ we first identify $\Aut(G;C)$, and then we find all subgroups $N\in\calA_{\max}(G;C)$, that is to say, all proper normal subgroups $N$ of $G$ which are maximal subject to being  $\Aut(G;C)$-invariant. It follows from \eqref{eqautd} and the definitions of the $C_i$ that $\Aut(G;C_1)=\la \tau_{n-1}\ra\cong Z_2$,
and $\Aut(G;C_2)=\la \tau_{n-1}\vp\ra\cong Z_2$. In both cases the subgroup  $M=\la a\ra$ is invariant under $\Aut(G)$, and hence $M\in \calA_{\max}(G;C_i)$ so $\Phi(G;C_i)\leq M$ for each $i$.
Let $N$ be an arbitrary subgroup in $\calA_{\max}(G;C_i)$. Since $N$ is a proper subgroup of $G$ it is contained in some maximal subgroup, namely either $M$, or one if the $M_p$ or $M_2'$ (with $n$ even in the last case).  If $N\leq M$, then by maximality $N=M$. If $N\leq M_p$ for some odd prime $p$, then, as $N$ is normal in $G$, it follows that $N$ is contained in $\cap_{g\in G}M_p^g=\la a^p\ra$, and hence $N$ is a proper subgroup of $M$, contradicting the maximality of $N$. Thus
\[
\mbox{if $n$ is odd then, for each $i$, $\calA_{\max}(G;C_i)=\{\la a\rangle\}$ and $\Phi(G;C_i)=\la a\ra > \Phi(G)$}.
\]
Suppose now that $n$ is even. We showed above that either $N=M$ or $N\leq M_2$ or $N\leq M_2'$. If $i=1$ then each of $M_2, M_2'$ is left invariant by $\Aut(G;C_1)=\la \tau_{n-1}\rangle$ and so
\[
\mbox{if $n$ is even then $\calA_{\max}(G;C_1)=\{\la a\rangle, M_2, M_2'\}$ and hence $\Phi(G;C_1)
=\la a^2\ra \geq  \Phi(G)$},
\]
and in fact $\Phi(G) < \Phi(G;C_1)$ unless $n=2^e\geq4$,  in which case $\Phi(G;C_1)=\Phi(G)=\la a^2\ra>1$. If $i=2$ and $N$ is contained in one of $M_2, M_2'$, then since $\Aut(G;C_2)=\la \tau_{n-1}\vp\rangle$ interchanges $M_2$ and $M_2'$, we have $N\leq M_2\cap M_2'=\la a^2\ra < M$, contradicting the maximality of $N$. Thus
\[
\mbox{if $n$ is even then $\calA_{\max}(G;C_2)=\{\la a\rangle\}$ and $\Phi(G;C_2)
=\la a\ra >  \Phi(G)$}.
\]
}
\end{example}

\begin{example}\label{ex:frat2}
{\rm
Let $G$ be an almost simple group with unique minimal normal subgroup $T$, a nonabelian simple group. If $G=T$, then the only proper normal subgroup of $G$ is the trivial subgroup and hence $\Phi(G)=\Phi(G;C)=1$ for all inverse-closed generating sets $C$.

Suppose now that $G\ne T$. For example $(G, T)=(S_5, A_5)$. Then $G$ has a maximal proper subgroup $M$ which does not contain $T$. (This assertion follows from the finite simple group classification.)
Thus $\Phi(G)\leq \cap_{g\in G}M^g=1$, that is, $\Phi(G)=1$. On the other hand, each maximal proper normal subgroup $N$ contains $T$, and for any generating set $C$ for $G$, the subgroup $T$ is $\Aut(G; C)$-invariant. Thus  $\Phi(G;C)\geq T$ and so  $\Phi(G;C)>\Phi(G)$.
Often we will have
$\calA_{\max}(G;C)=\{T\}$ so that $\Phi(G;C)= T$, but this is not necessarily the case. For example if $G= {\rm P}\Gamma{\rm L}(2,2^{p^2})$ with $T={\rm PSL}(2,2^{p^2})$, for some odd prime $p$, then  $\calA_{\max}(G;C)=\{T.p\}$ and $\Phi(G;C)= T.p$.
}
\end{example}

We now prove Theorem~\ref{thm:frat}.

\begin{proof} \emph{of Theorem}~\ref{thm:frat}.  Let $G=\la C\ra$. \\
(a)  Let $X$ be a normal $C$-closed generating set of $G$, and let $Y=X\setminus(X \cap \Phi(G;C))$. By \eqref{eq:phi}, $\Phi(G;C)$ is normal $C$-closed, and hence also $Y$ is normal $C$-closed. This implies that $K:=\la Y\ra$ is a normal $\Aut(G;C)$-invariant subgroup of $G$. Suppose that $K\ne G$. Then $K\in\calA(G;C)$, and there exists $N\in\calA_{\max}(G;C)$ such that $K\leq N$. By definition $Y\subseteq K\subseteq N$, and also $X\cap \Phi(G;C)\subseteq \Phi(G;C) \subseteq N$ by \eqref{eq:phi}. Thus $X\subseteq N$, which is a contradiction since $X$ generates $G$. Thus $K=G$, and hence $Y$ is a   normal $C$-closed generating set.

(b) Let $S$ be the subset consisting of all elements $y$ of $G$ which satisfy the condition in Theorem~\ref{thm:frat}(b). Then by part (a), $\Phi(G;C)\subseteq S$. Now let $y\in S$ and let
$N\in \calA_{\max}(G;C)$. We claim that $y\in N$. Suppose to the contrary that $y\not\in N$.
Let $\widehat{N} = N\cup \{(y^g)^\sigma\mid \sigma\in\Aut(G;C), g\in G\}$ (the normal $C$-closure of $N\cup\{y\}$), and let $M=\la \widehat{N}\ra$. Since $N$ is normal in $G$ and
$\Aut(G;C)$-invariant, the subgroup $M$  is also  normal in $G$ and $\Aut(G;C)$-invariant. Then since $N$ is properly contained in $M$ and $N\in \calA_{\max}(G;C)$, it follows that
$G=M$. Then, since $y$ satisfies  the condition in Theorem~\ref{thm:frat}(b), the normal $C$-closure of $N$ (which is just $N$ itself) must generate $G$. This is a contradiction. Hence $y\in N$. Since this holds for all $N\in \calA_{\max}(G;C)$ it follows from \eqref{eq:phi} that $y\in\Phi(G;C)$.
 \end{proof}

\subsection{Transitive inverse-closed generating sets}

For our application to normal edge-transitive Cayley graphs, we are interested in generating sets $C$ for $G$ which have the following property (see \cite[Lemma 1]{praeger}).

\begin{definition}\label{def:trans}
{\rm
An inverse-closed generating set $C$ for a nontrivial group $G$ is said to be \emph{transitive} if either
the $\Aut(G;C)$-action on $C$ is transitive, or there are two orbits $C_+$ and $C_-$ in this action such that $C=C_+\cup C_-$ and $C_-=(C_+)^{-1}$, where $(C_+)^{-1}  = \{c^{-1}\mid c\in C_+\}$.
}
\end{definition}

Transitive generating sets have some useful properties.

\begin{lemma}\label{lem:notinphi}
	Let $G$ be a nontrivial  finite group with transitive inverse-closed generating set  $C$. Then
	\begin{enumerate}
	\item[(a)]  $C\cap\Phi(G;C)=\emptyset$;
	\item[(b)]  there is a positive integer $\ell$ such that $|C\cap y\Phi(G;C)|\in\{0,\ell\}$ for each $y\in G$. Moreover, if both $C\cap y\Phi(G;C)\ne\emptyset$ and $C\cap y'\Phi(G;C)\ne\emptyset$, then there exists $\sigma\in\Aut(G;C)$ such that $(y'\Phi(G;C))^\sigma$ is either $y\Phi(G;C)$ or
	$y^{-1}\Phi(G;C)$.
\end{enumerate}
\end{lemma}

\begin{proof}
(a)  Suppose that $c\in C\cap\Phi(G;C)$. Then, since $\Phi(G;C)$ is $\Aut(G;C)$-invariant, each element of the orbit containing $c$, under the natural action of $\Aut(G;C)$, lies in $\Phi(G;C)$. By
Definition~\ref{def:trans}, this means that for each $c'\in C$, either $c'$ or its inverse lies in $\Phi(G;C)$. Then since $\Phi(G;C)$ is closed under forming inverses (as it is a subgroup), we have  $C\subseteq \Phi(G;C)$. Hence $G=\la C\ra\leq \Phi(G;C)$, which is a contradiction, since $\Phi(G;C)$ is a proper subgroup.

(b)  Suppose that $C\cap y\Phi(G;C)\ne\emptyset$ and $C\cap y'\Phi(G;C)\ne\emptyset$. Then there exist $c, c'\in C$ such that $y\Phi(G;C)=c\Phi(G;C)$ and $y'\Phi(G;C)=c'\Phi(G;C)$. By
Definition~\ref{def:trans},  there exists $\sigma\in\Aut(G;C)$ such that $(c')^\sigma$ is equal to either $c$ or
$c^{-1}$, and hence, since $\Phi(G;C)$ is $\Aut(G;C)$-invariant,  $(y'\Phi(G;C))^\sigma =(c'\Phi(G;C))^\sigma$ is either $c\Phi(G;C)=y\Phi(G;C)$ or
$c^{-1}\Phi(G;C)=y^{-1}\Phi(G;C)$. Note that, since $\Phi(G;C)$ is a normal subgroup of $G$, we have $(c\Phi(G;C))^{-1} = c^{-1}\Phi(G;C)$. Then, since $C$ is inverse-closed, $(C\cap c\Phi(G;C))^{-1} = C^{-1}\cap (c\Phi(G;C))^{-1} = C \cap  c^{-1}\Phi(G;C)$.
Thus, the cardinality $|C\cap y'\Phi(G;C)|=|C\cap y\Phi(G;C)|$ is constant for all non-empty intersections.
\end{proof}


 \section{Quotients of normal edge-transitive graphs}\label{sec:netg}

In this section we prove Theorem~\ref{main decomposition}, as a consequence of a  more general result, Theorem~\ref{main}. First we summarise our notation for graphs and their homomorphisms.

\subsection{Graph concepts}\label{sub:graphnot}
We will only consider finite simple undirected  graphs. That is, a \emph{graph} $\Gamma=(V, E)$ will consist of a finite set $V$ of vertices, and a set $E$ of edges, where each edge is an unordered pair of distinct vertices. We often denote the vertex set and the edge set of $\Gamma$ by $V(\Gamma)$ and $E(\Gamma)$, respectively.

A \emph{subgraph} $\Gamma'=(V', E')$ of  $\Gamma=(V, E)$, is a graph with $V'\subseteq V$, $E'\subseteq E$, and each $\{u,v\}\in E'$ is a subset of $V'$.  A \emph{homomorphism} from $\Gamma_1=(V_1,E_1)$ to $\Gamma_2=(V_2,E_2)$ is a map $f:V_1\to V_2$  such that, for each $\{u,v\}\in E_1$, the image $\{u^f, v^f\}\in E_2$. A graph homomorphism $f$ is called
\begin{itemize}
	\item  an {\it epimorphism}, if $f$ is onto (on vertices);
	\item  a {\it full homomorphism}, if the induced map $E_1\to E_2$ is onto;
	\item  a {\it full epimorphism}, if $f$ is both an epimorphism and a full homomorphism;
	\item an {\it isomorphism} if $f$ is a full homomorphism and is a bijection on vertices; in this case  $\Gamma_1$ and $\Gamma_2$ are said to be \emph{isomorphic}, written $\Gamma_1\cong \Gamma_2$.
\end{itemize}
For graphs $\Ga_i=(V_i,E_i)$, where $1\leq i\leq n$, their \emph{direct product}
is the graph $(V,E)$ with $V=V_1\times\dots\times V_n$ and $E$ consisting of all the pairs $\{u,v\}\subset V$ such that $\{u_i,v_i\}\in E_i$ for each $i\leq n$. We denote this graph by $\prod_{i=1}^n\Ga_i$. A \emph{subdirect product} of a family of graphs $\{\Gamma_1,\dots,\Ga_n\}$ is a subgraph $\Gamma$ of the direct product $\prod_{i=1}^n \Gamma_i$ such that, for each $i\leq n$, the natural map $\pi_i:\Gamma\to \Gamma_i$ (where $\pi_i:u\to u_i$) is a graph epimorphism. It is called a \textit{full subdirect product} if each $\pi_i$ is a full epimorphism (see \cite{Caicedo}).

Our main graph theoretical result Theorem~\ref{main decomposition} relies in an essential way on the following fact about products of Cayley graphs.

\begin{lemma}\label{lem:cayprod}
Let $n$ be a positive integer $n\geq2$, and let $G_1,\dots,G_n$ be nontrivial finite groups with inverse-closed generating sets $C_1,\dots,C_n$, respectively.
Consider the Cayley graphs  
\[
\Gamma:=\Cay(G_1\times\dots\times G_n, C_1\times\dots\times C_n),\ \text{and}\ \Gamma_i:=\Cay(G_i, C_i),
\ \text{for}\ i=1,\dots,n.
\] 
Then $\Gamma$ is the direct product graph $\prod_{i=1}^n\Gamma_i$.
\end{lemma}

\begin{proof}
By the definition of a Cayley graph, for each $i$, the edge set of $\Gamma_i$ is the set of all pairs of the form $\{g_i,c_ig_i\}$ with $g_i\in G_i$ and $c_i\in C_i$. Similarly, the edge-set of $\Gamma$ is the set of all pairs of the form $\{g, cg\}$ where $g=(g_1,\dots,g_n)\in\prod_{i=1}^nG_i$ and $c=(c_1,\dots,c_n)\in\prod_{i=1}^nC_i$. In such a pair the element $cg=(c_1g_1,\dots,c_ng_n)$, and hence these pairs $\{g,cg\}$ are precisely the pairs for which $\{g_i,c_ig_i\}$ is an edge of $\Gamma_i$, for each $i=1,\dots,n$. By the definition of a direct product graph, $\Gamma$ is the direct product $\prod_{i=1}^n\Gamma_i$.  
\end{proof}

\subsection{Proof of Theorem~\ref{main decomposition}}
Now we state and prove our main decomposition result for normal edge-transitive Cayley graphs. Recall that, for an inverse-closed generating set $C$ for a group $G$,  $\Cay(G;C)$ is normal edge-transitive if and only if $G_R\rtimes\Aut(G;C)$ is transitive on edges and, by \cite[Proposition 1(b)]{praeger}, this is equivalent to the set  $C$ being transitive in the sense of Definition~\ref{def:trans}.

\begin{theorem}\label{main}
	Let $G$ be a nontrivial  finite group with inverse-closed generating set  $C$ such that
	$\Ga=\Cay(G;C)$ is normal edge-transitive. Then
	there exist $k\geq 1$ and $N_1,\dots,N_k\in\calA_{\max}(G;C)$ for which the following conditions hold for the subgroup $X:=\cap_{i=1}^k N_i$
	and quotient graphs $\Ga_i:= \Cay(G/N_i;CN_i/N_i)$.
	\begin{enumerate}[{\rm (i)}]
		\item $X\in\calA(G;C)$ and the map $\zeta: gX\rightarrow (gN_1,\cdots,gN_k)$ is a group isomorphism $G/X \rightarrow \prod_{i=1}^k G/N_i$;
		\item  for each $i\leq k$, $\Ga_i\in \caycs$, and $\zeta$ induces a graph isomorphism from $\Cay(G/X; CX/X)$ onto a full subdirect product of $\prod_{i=1}^k\Ga_i$.
\item[{\rm (iii)}] Moreover the unique smallest $X$ for which both (i) and (ii) hold is $X=\Phi(G;C)$.
	\end{enumerate}
\end{theorem}

\begin{proof}
We prove parts (i) and (ii) by induction on $|G|$. These assertions hold if  $\calA_{\max}(G;C)=\{ \{1_G\}\}$, since in this case $G$ is characteristically simple and so $\Cay(G;C)\in\caycs$, and we take $k=1$ and $X=N_1=\Phi(G;C)=1$. In particular this is the case if $G$ is simple. Suppose inductively that $\calA_{\max}(G;C)$ consists of nontrivial subgroups, and that parts (i) and (ii) hold for groups of order less than $|G|$.

Choose $N_1\in\calA_{\max}(G;C)$. By definition $G/N_1$ is characteristically simple and since the graph $\Ga_1:= \Cay(G/N_1;CN_1/N_1)$ is normal edge-transitive by
\cite[Theorem 3(c)]{praeger}, we have $\Ga_1\in\caycs$.
Let $H\in\calA(G;C)$ and suppose that $N_1H\ne G$. Then $N_1H$ is a proper normal subgroup of $G$ and is   $\Aut(G;C)$-invariant, and hence $N_1H\in\calA(G;C)$. By the maximality of $N_1$, it follows that $N_1H=N_1$, that is to say, $H\leq N_1$. If this is true for all $H\in\calA(G;C)$, then $N_1$ is the unique maximal element of $\calA(G;C)$, and hence  $\calA_{\max}(G;C)=\{N_1\}$, $\Phi(G;C)=N_1$, and parts (i) and (ii) hold with $k=1, X=\Phi(G;C)=N_1$.

Thus we may assume that there exists  $H\in\calA(G;C)$ such that $N_1H= G$. Since
$N_1\ne G$, we have $H\ne1$ so $|G/H|<|G|$, and by induction parts (i) and (ii) hold for  $G/H$ with generating set $CH/H$. Hence there exists $k\geq2$ and subgroups $\ov{N}_2,\dots,\ov{N}_k\in\calA_{\max}(G/H;CH/H)$ such that parts (i)--(ii) hold with
$$
\ov{G}=G/H,\ \ov{C}=CH/H,\ k-1,\ \ov{N}_i,\ \ov{K}=\cap_{i=2}^k\ov{N}_i,\ \ov{\Ga}_i=\Cay(\ov{G}/\ov{N}_i; \ov{C}\ov{N}_i/\ov{N}_i),
$$
and  $\ov{\zeta}:\ov{G}/\ov{K}\rightarrow \prod_{i=2}^k \ov{G}/\ov{N}_i$
in place of $G, C, k, N_i, X, \Ga_i, \zeta$, respectively.
By induction, $\ov{\zeta}$ induces a graph isomorphism from $\Cay(\ov{G}/\ov{K}; \ov{C}\ov{K}/\ov{K})$ onto a full subdirect product of $\prod_{i=2}^k\ov{\Gamma}_i$.
Note that each $\ov{N}_i$ is of the form $N_i/H$ for a unique $N_i\in\calA_{\max}(G;C)$ such that $H\leq N_i$.  Moreover, $\ov{G}/\ov{N_i}\cong G/N_i$ and the natural map $\nu_i:xN_i\to \ov{x}\ov{N}_i$, where $\ov{x}=xH\in \ov{G}=G/H$, defines an isomorphism $G/N_i\to \ov{G}/\ov{N_i}$. If $x\in CN_i$, then
$\ov{x}=xH\in (CN_i)/H=\ov{CN_i}$, and hence
$(CN_i/N_i)\nu_i= \ov{CN_i}/\ov{N_i}$. It follows that $\nu_i$ induces a graph isomorphism from  $\Ga_i:=\Cay(G/N_i,  CN_i/N_i)$ to $\ov{\Ga}_i$. In particular $\Ga_i\in \caycs$ for $2\leq i\leq k$. Moreover the map
\[
\nu:(\ov{x}_2\ov{N}_2,\dots,\ov{x}_k\ov{N}_k)\to (\ov{x}_2\ov{N}_2\nu_2^{-1},\dots,\ov{x}_k\ov{N}_k\nu_k^{-1})=(x_2N_2,\dots,x_kN_k)
\]
is a group isomorphism $\nu: \prod_{i=2}^k \ov{G}/\ov{N}_i\to  \prod_{i=2}^k G/N_i$.
Setting $K:=\cap_{i=2}^k N_i$, we have $\ov{K}=K/H\in \calA(\ov{G};\ov{C})$ and $K\in\calA(G;C)$, and a natural isomorphism $\mu: G/K\to \ov{G}/\ov{K}$ given by $\mu: xK\to \ov{x}\ov{K}$.
The composition $\zeta':= \mu\circ \ov{\zeta}\circ \nu$ is a group homomorphism  $\zeta': G/K\to \prod_{i=2}^k G/N_i$ such that $\zeta':xK \rightarrow (xN_2,\dots,xN_k)$ for each $x\in G$.

\medskip\noindent
\emph{Proof of part (i).}  Let $X :=\cap_{i=1}^k N_i = N_1\cap K$. Then since each $N_i\in\calA_{\max}(G;C)$ we have $X\in\calA(G;C)$. Also since $G=N_1H$ and $H\leq K$, we have $G=N_1K$. The quotient $G/X$ has normal subgroups $K/X$ and $N_1/X$ that intersect trivially, and hence $G/X\cong K/X\times N_1/X$. Since also $K/X\cong KN_1/N_1=G/N_1$ and $N_1/X\cong N_1K/K= G/K$, it follows that the map $\zeta_1:xX \rightarrow (xN_1, xK)$ is an isomorphism from $G/X$ to $G/N_1\times G/K$. Now the map $\zeta_2: G/N_1\times G/K\to\prod_{i=1}^k G/N_i$ given by $(xN_1, xK)\zeta_2= (xN_1, (xK)\zeta')$ is also an isomorphism, and therefore the composition $\zeta:=\zeta_1\circ \zeta_2 : G/X\rightarrow \prod_{i=1}^k G/N_i$ is the natural isomorphism required in part (i).

\medskip\noindent
\emph{Proof of part (ii).} We have seen already that each $\Ga_i\in\caycs$. Now we prove the rest of part (ii). The map  $\zeta$ induces a graph isomorphism from
$\Sigma:=\Cay(G/X, CX/X)$ to $\Cay(\prod_{i=1}^k G/N_i,  (CX/X)\zeta)$.  To determine $(CX/X)\zeta$, note that, since $X=\cap_{i=1}^kN_i$ is contained in each $N_i$, for  each $c\in C$ and $x\in X$ the image $(cxX)\zeta = (cN_1,\dots,cN_k)$, and thus
$(CX/X)\zeta\subseteq \prod_{i=1}^k CN_i/N_i$. Hence $\Sigma\zeta$ is a subgraph of  $\Cay(\prod_{i=1}^k G/N_i, \prod_{i=1}^k CN_i/N_i)$ which, by Lemma~\ref{lem:cayprod}, is isomorphic to the direct product graph $\prod_{i=1}^k\Cay(G/N_i,CN_i/N_i)=\prod_{i=1}^k\Ga_i$. Moreover it follows from the definition of $\Sigma$ that the image of $\Sigma\zeta$ under the $j^{th}$ projection map $\pi_j:\prod_{i=1}^k\Ga_i\rightarrow \Ga_j$ is equal to $\Ga_j$ so $\Sigma\zeta$ is a subdirect product of $\prod_{i=1}^k\Ga_i$. Finally each edge of $\Ga_j$ is of the form $e=\{gN_j, cgN_j\}$ for some $c\in C$, $g\in G$, and so $e$ is the image under $\zeta\pi_j$ of the edge $\{gX, cgX\}$ of $\Sigma$. Thus $\Sigma\zeta$ is a full subdirect product and part (ii) holds.

Therefore parts (i) and (ii) have been proved by induction. It remains to prove part (iii).


\smallskip
Proof of part (iii). From all the sets of subgroups $N_1,\dots, N_k\in\calA_{\max}(G;C)$ such that parts (i) and (ii) hold with  $X=\cap_{i=1}^k N_i$ and
$\zeta:gX\rightarrow (gN_1,\dots,gN_k)$ and the $\Gamma_i=\Cay(G/N_i, CN_i/N_i)
\in\caycs$,  choose the subgroups $N_1,\dots,N_k$ such that $|X|$ is as small as possible. We claim that $X=\Phi(G;C)$. Note that, by the definitions of $\Phi(G;C)$ and $X$, it follows that $X\in\calA(G;C)$ and $\Phi(G;C)\leq X$. To prove the claim it is therefore sufficient to show that each $N\in\calA_{\max}(G;C)$ contains $X$, since this will imply that $\Phi(G;C)\geq X$ and hence $\Phi(G;C)=X$.

Suppose to the contrary that there exists $N\in\calA_{\max}(G;C)$ such that $X\not\leq N$. Then $NX$ properly contains $N$ and by the maximality of $N$,  we have $G=NX$.
Let $Y=N\cap X$. Then the same argument as in the proof of part (i) shows that $G/Y\cong N/Y\times X/Y\cong G/X\times G/N$, and that the map $\zeta_1: gY\rightarrow (gX, gN)$ is an isomorphism from $G/Y$ to $G/X\times G/N$; and further, that the composition of $\zeta_1$ followed by the map $(gX, gN)
\rightarrow ((gX)\zeta, gN)$ from $G/X\times G/N$ to $(\prod_{i=1}^kG/N_i)\times G/N$ is the natural isomorphism $\zeta'$ from $G/Y$ to $\prod_{i=1}^{k+1} G/N_i$, where $N_{k+1}:=N$. Note that $\Gamma_{k+1}:=\Cay(G/N_{k+1}, CN_{k+1}/N_{k+1})\in\caycs$ since $N\in\calA_{\max}(G;C)$. Then the argument in the proof of part (ii) shows that $\zeta'$ induces an isomorphism from $\Cay(G/Y, CY/Y)$ onto a full subdirect product of $\prod_{i=1}^{k+1}
\Ga_i$. Thus parts (i) and (ii) hold for $N_1,\dots, N_{k+1}\in\calA_{\max}(G;C)$, and we have $Y=\cap_{i=1}^{k+1}N_i = X\cap N_{k+1}$, a proper subgroup of $X$. This contradicts the minimality of $|X|$. Hence   each $N\in\calA_{\max}(G;C)$ contains $X$, and therefore $X=\Phi(G;C)$.
\end{proof}

Finally in this subsection we prove Theorem \ref{main decomposition}. \\

\begin{proof} \emph{of Theorem}~\ref{main decomposition}. 
By Theorem~\ref{main}(iii), for the group $X=\Phi(G;C)$, parts (i) and (ii) of  Theorem~\ref{main} hold. This yields everything we need for Theorem \ref{main decomposition} except for the assertion that $\calA(G/N_i, CN_i/N_i)=\{\{1\}\}$ for each $i$. To prove this last assertion, recall that each of the subgroups $N_i\in\calA_{\max}(G;C)$. Suppose to the contrary that, for some $i$, there exists a nontrivial $M_i\in \calA(G/N_i, CN_i/N_i)$.
Then $M_i = M/N_i$ where $M\unlhd G$, $N_i<M< G$,
and $M/N_i$ is invariant under $\Aut(G/N_i; CN_i/N_i)$. In particular $M/N_i$ is invariant under the subgroup of  $\Aut(G/N_i; CN_i/N_i)$ induced by the action of $\Aut(G;C)$, and hence $M$ is $\Aut(G;C)$-invariant. This implies that $M\in\calA(G;C)$, contradicting the maximality of $N_i$. Thus each $\calA(G/N_i, CN_i/N_i)=\{\{1\}\}$.
\end{proof}

\subsection{Transitive inverse-closed generating sets for a given group}

Inverse-closed generating sets for a group $G$, that are transitive in the sense of Definition~\ref{def:trans}, are often fairly natural subsets. For example we may take a full conjugacy class of involutions in a nonabelian simple group $G$. \emph{Let $\mathcal{T}(G)$ denote the set of all transitive, inverse-closed generating sets for $G$.} First we observe some simple properties of $\mathcal{T}(G)$.

\begin{lemma}\label{lem:tg}
Let $G, H$ be nontrivial finite groups. Then
\begin{enumerate}
\item[(a)] $\mathcal{T}(G)$ is $\Aut(G)$-invariant; and
\item[(b)] if $C\in\mathcal{T}(G)$ and $D\in\mathcal{T}(H)$, and if either $\Aut(G;C)$ is transitive on $C$ or $\Aut(H;D)$ is transitive on $D$, then $C\times D
\in\mathcal{T}(G\times H)$.
\end{enumerate}
\end{lemma}

\begin{proof}
(a)  It is easy to check that, for $\sigma\in\Aut(G)$ and  $C\in\mathcal{T}(G)$, the image
$C^\sigma$ is an inverse-closed generating set for $G$. Further, $C^\sigma$ is $\Aut(G;C^\sigma)$-invariant (since $C$ is $\Aut(G;C)$-invariant), and the fact that $C$ is transitive implies that $C^\sigma$ is transitive. Thus $C^\sigma\in\mathcal{T}(G)$.

(b) Suppose that $C\in\mathcal{T}(G)$ and $D\in\mathcal{T}(H)$. Then $C\times D$ is inverse closed, and generates $G\times H$. Also $\Aut(G;C)\times\Aut(H;D)\leq \Aut(G\times H; C\times D)$. If both $\Aut(G;C)$ is transitive on $C$ and $\Aut(H;D)$ is transitive on $D$, then $\Aut(G;C)\times\Aut(H;D)$ is transitive on $C\times D$, and hence  $C\times D
\in\mathcal{T}(G\times H)$. Suppose next that $\Aut(G;C)$ is transitive on $C$ and  $\Aut(H;D)$ has two orbits on $D$, namely $D_0$ and $D_0^{-1}$. Then $C\times D_0$ and $C\times D_0^{-1}$
are the two orbits of   $\Aut(G;C)\times\Aut(H;D)$ in $C\times D$, and $(C\times D_0)^{-1} = C\times D_0^{-1}$ since $C$ is inverse-closed. So again  $C\times D
\in\mathcal{T}(G\times H)$, and the same argument works in the case where $\Aut(H;D)$ is transitive on $D$ but $\Aut(G;C)$ is not transitive on $C$.
\end{proof}

We note that the argument in the proof of part (b) does not extend to the case where both  $\Aut(G;C)$ is not transitive on $C$ and  $\Aut(H;D)$ is not transitive on $D$, since in this case $\Aut(G;C)\times\Aut(H;D)$ has four orbits in $C\times D$.

We now give a family of examples to illustrate that a group may have several transitive inverse-closed generating sets, and in particular that some may be properly contained in others.

\begin{example}\label{ex:frat3}
{\rm
Let $G=\la a, b\mid a^n=b^2=1, bab=a^{-1}\ra$, the dihedral group $D_{2n}$ of order $2n\geq6$, as in \eqref{eqd}, and note that $\Aut(G)$ is given in \eqref{eqautd}. Consider the inverse-closed generating set $C_2=\{b,ba\}$ as in Example~\ref{ex:frat}.

\begin{enumerate}
\item[(a)]  As we saw  in Example~\ref{ex:frat}, $\Aut(G;C_2)= \langle \tau_{n-1}\vp\rangle \cong Z_2$ (with $\tau_{n-1}, \vp$ as in \eqref{eqvpt}), and $\tau_{n-1}\vp$ interchanges the two involutions in $C_2$. Hence $C_2$ is a transitive inverse-closed generating set for $G$, that is, $C_2\in\mathcal{T}(G)$.

\item[(b)] By Lemma~\ref{lem:tg}(a), for each $i=0,\dots,n-1$, the set $D_i:= C_2^{\vp^i}=\{ ba^i, ba^{i+1}\}\in\mathcal{T}(G)$. We note that each of these $n$ subsets $D_i$ has size $2$, and each of the $n$ non-central involutions $ba^j$ lies in exactly two of them.

\item[(c)] A larger natural subset lying in $\mathcal{T}(G)$ is $C'=\{ba^i\mid 0\leq i\leq n-1\}$,
the set of all non-central involutions in $G$: this set  $C'$  forms a single orbit of $\Aut(G)$, and hence $\Aut(G;C')=\Aut(G)$ and $C'\in\mathcal{T}(G)$.

\item[(d)] The generating sets in (b) and (c) correspond to quite different normal edge-transitive Cayley graphs: each $\Cay(G;D_i) \cong C_{2n}$ (a cycle on $2n$ vertices), while $\Cay(G;C') \cong K_{n,n}$ (a complete bipartite graph).
\end{enumerate}
}
\end{example}

\section{Transitive inverse-closed generating sets for dihedral groups}\label{sec:dehedral}

It is instructive to take the analysis in Section~\ref{sec:netg} further and characterise the set $\mathcal{T}(G)$ for groups $G$ in some specific infinite families.
We do this in Theorem~\ref{lem:d2n} for the family of dihedral groups, which in particular helps place the graphs in Example~\ref{ex:frat3} into a broader context. We note that Talebi~\cite[Theorem 3.3]{Talebi} showed that the only normal edge-transitive Cayley graph $\Cay(D_{2n};C)$ of valency $n$ is the graph in Example~\ref{ex:frat3}(d): this example arises in part (c) of Theorem~\ref{lem:d2n}.

\begin{theorem}\label{lem:d2n}
Let $G=\la a, b\mid a^n=b^2=1, bab=a^{-1}\ra$, the dihedral group $D_{2n}$ of order $2n\geq6$, as in $\eqref{eqd}$, with $\Aut(G)$ as in $\eqref{eqautd}$. Let $C\in\mathcal{T}(G)$, $A_0=\Aut(G;C)$, and $A_0\cap \langle\vp\rangle = \la \vp^r\rangle$ where $r$ divides $n$.

\begin{enumerate}
\item[(a)] Then there exists a non-empty subset $I\subseteq \{0,1,\dots,r-1\}$ such that $C=\cup_{i\in I} ba^i\langle a^r\rangle$, and $A_0/\langle\vp^r\rangle$ acts transitively on the set $\mathcal{S}=\{ba^i\langle a^r\rangle | i\in I\}$.

\item[(b)] Up to isomorphism (that is replacing $C$ by $C^\sigma$ for some $\sigma\in \Aut(G)$), we may assume that $0\in I$.

\item[(c)] If $I=\{0\}$, then $r=1$, and $C$ is as in Example~$\ref{ex:frat3}$(c), with $A_0=\Aut(G)$.

\item[(d)] If $I=\{0,i\}$, for some $i\ne 0$, then $\gcd(i,r)=1$, there exists a positive integer $k<n$ such that $\gcd(k,n)=1$ and $r \mid (k+1)$, and $C= b\langle a^r\rangle \cup ba^i\langle a^r\rangle$ with $\langle\vp^r, \tau_k\vp^i\rangle$ a subgroup of $A_0$ acting transitively on $C$. In particular if $r=n$, then for some $\sigma\in\Aut(G)$, $C^\sigma = \{b,ba\}$  as in Example~$\ref{ex:frat3}$(a), with $A_0^\sigma=\langle \tau_{n-1}\vp\rangle\cong Z_2$.
\end{enumerate}
\end{theorem}

\begin{proof}
(a)  Recall the explicit description of $\Aut(G)$ given in \eqref{eqvpt} and \eqref{eqautd}. Let $C\in\mathcal{T}(G)$.
Since $C$ generates $G$, it follows that $C$ must contain an element of $G\setminus\la a\ra$,
the set of $n$ non-central involutions of $G$. Then since $C$ is a transitive inverse-closed generating set for $G$ and $C$ contains an involution, it follows from  Definition~\ref{def:trans} that
$A_0$ is transitive on $C$ so $C\subseteq G\setminus\la a\ra= b\langle a\rangle$.
Now $A_0$ contains $\langle\vp^r\rangle$ and the $\langle\vp^r\rangle$-orbits in $b\langle a\rangle$ are the cosets $ba^i\langle a^r\rangle$ for $0\leq i\leq r-1$. Thus $C=\cup_{i\in I} ba^i\langle a^r\rangle$ for some $I\subseteq \{0,1,\dots,r-1\}$. The normal subgroup   $\langle\vp^r\rangle$ of $A_0$ fixes each of these cosets setwise, and so, since $A_0$ is transitive on $C$, it follows that $A_0/\langle\vp^r\rangle$ acts transitively on the set $\mathcal{S}$, proving part (a).

(b) By Lemma~\ref{lem:tg}(a), the image of $C$ under any element of $\Aut(G)$ is again an element of $\mathcal{T}(G)$. Let $i\in I$. Then, by \eqref{eqvpt}, $\vp^{n-i}$ maps $ba^i$ to $b$ and hence maps $ba^i\langle a^r\rangle$ to $ba^0\langle a^r\rangle$. Thus we may assume that $0\in I$, proving (b).

(c) Suppose that $I=\{0\}$, that is, $C=b\langle a^r\rangle$. Since $C$ is a generating set for $G$ it follows that $G=\langle b, a^r\rangle$ and hence (since $r$ divides $n$) that $r=1$.  Thus  $C$ is the set of all non-central involutions in $G$,  as in Example~\ref{ex:frat3}(c), and $A_0=\Aut(G)$, proving (c).

(d) Suppose that $I=\{0,i\}$ for some $i\ne 0$, that is, $C= b\langle a^r\rangle \cup ba^i\langle a^r\rangle$.  Since $C$ is a generating set for $G$ it follows that $G=\langle b, a^r, a^i\rangle$ and hence $a\in\langle a^r, a^i\rangle$. This implies that $\gcd(i,r)=1$. Now $A_0$ is transitive on $C$, and since $\langle\vp^r\rangle$ fixes each $\langle a^r\rangle$-coset setwise, it follows that there exist $k, j$ with $1\leq k\leq n-1$, $\gcd(k,n)=1$, and $0\leq j\leq r-1$, such that $\tau_k\vp^j\in A_0\setminus \langle \vp^r\rangle$ and $\tau_k\vp^j$  interchanges $ b\langle a^r\rangle$ and $ba^i\langle a^r\rangle$. By \eqref{eqvpt}, $\tau_k\vp^j: b\to ba^j$ and hence $j=i$ (recall $0\leq j<r$). Also $\tau_k\vp^i:ba^i\to ba^{i(k+1)}$ and hence $ba^{i(k+1)}\in  b\langle a^r\rangle$. This implies that $r$ divides $i(k+1)$ and since $\gcd(i,r)=1$ this means that $r \mid (k+1)$. If these conditions on $k$ are satisfied then  $A_1:=\langle\vp^r, \tau_k\vp^i\rangle$
is a subgroup of $A_0$ that acts transitively on $C$. Finally if $r=n$, then $C=\{b,ba^i\}$ and $\gcd(i,n)=1$. Thus there exists $\ell$ such that $1\leq \ell\leq n-1$ and $i\ell\equiv 1\pmod{n}$, and hence $\tau_\ell\in T<\Aut(G)$ and $\tau_\ell :b\to b, ba^i\to ba$. Therefore $C^{\tau_\ell}=\{b,ba\}$,  as in Example~\ref{ex:frat3}(a). The subgroup of index 2 in $A_0$ which fixes $C$ pointwise must fix each of $a, b$ and hence must be trivial. It follows that $A_0\cong Z_2$, and in fact $A_0^{\tau_\ell} = \Aut(G;C^{\tau_\ell})=\langle \tau_{n-1}\vp\rangle$.
\end{proof}

\subsection{Four-valent normal edge-transitive Cayley graphs of dihedral groups}\label{sub:val4}

For the special case where $n$ is an odd prime, Talebi \cite[Theorem 3.5]{Talebi} classified the $4$-valent normal edge-transitive Cayley graphs of $D_{2n}$, or equivalently, the  subsets of $\mathcal{T}(D_{2n})$ of cardinality $4$. We will complete this classification for arbitrary $n$.  Recall \eqref{eqd}-\eqref{eqautd}, and in particular
$$
G=\la a, b\mid a^n=b^2=1, bab=a^{-1}\ra
$$
and $\Aut(G)=A\rtimes T$ with generators as in \eqref{eqautd}, so that, for $k$ with $\gcd(k,n)=1$ and all $i, j$,
\begin{equation}\label{eq-gen}
(a^i)^{\tau_k\varphi^j}=a^{ik},\quad  (ba^i)^{\tau_k\varphi^j}=ba^{ik+j}.
\end{equation}
Let $C\in\mathcal{T}(G)$ of cardinality $|C|=4$, so that $\Gamma:=\Cay(G,C)$ is a $4$-valent
normal edge-transitive Cayley graph for $G$. All such graphs are of this form. By Theorem~\ref{lem:d2n}(b), up to isomorphism we may (and shall) assume that $b\in C$.
In \cite[Conjecture, p.451]{Talebi} Talebi  conjectures that each example must belong to one of two families which we describe in Example~\ref{ex:talebi}. (Note that there are two misprints in line 3 of
\cite[Conjecture, p.451]{Talebi}, namely the symbol $p$ should be $n$, and  the second family should be the one given in \cite[Theorem 3.4]{Talebi}.)

\begin{example}\label{ex:talebi}
Let $G=D_{2n}$ and $\Aut(G)$ be as above, with $n\geq3$, and define $C$ as in (a) or (b).
\begin{enumerate}
\item[(a)] $C=\{b, ba, ba^i, ba^{1-i}\}$, where $2\leq i\leq n-1$ such that $\gcd(2i-1, n)=1$ and $2i(i-1)\equiv 0\pmod{n}$;
\item[(b)] $C=\{b, ba, ba^{k+1}, ba^{k^2+k+1} \}$, where $1\leq k\leq n-2$, $\gcd(k,n)=1$, and $1+k+k^2+k^3\equiv 0\pmod{n}$.
\end{enumerate}

\end{example}

We establish some properties of these examples. In particular we decide when a subset $C$ arises in more than one way in Example~\ref{ex:talebi}. The fact that $C\in\mathcal{T}(G)$ for these subsets is proved in \cite[Theorem 3.4]{Talebi} for $C$  in case (a), and in \cite[Theorem 3.5]{Talebi} for $C$ in case (b) (but only if $n$ is prime) by exhibiting certain elements of $\Aut(G;C)$. We determine the subgroups $\Aut(G;C)$ completely, using the notation in \eqref{eq-gen}.

\begin{proposition}\label{prop:talebi}
Let $n, G, C$ be as in Example~$\ref{ex:talebi}$ (a) or (b), and let $A_0=\Aut(G;C)$. Then

\begin{enumerate}
\item[(a)] $|C|=4$ and $C\in\mathcal{T}(G)$. Moreover,
\begin{enumerate}
\item[(i)] if $C$ is as in Example~$\ref{ex:talebi}$ (a), then
\begin{equation*}
A_0= \begin{cases}
\langle \tau_{n/2+1}, \tau_{n-1}\vp\rangle \cong D_8 & \text{if $i= n/2$ or $n/2+1$, and $n\equiv 0\pmod{4}$,}\\
\langle \tau_{n-1}\varphi, \tau_{2i-1}\varphi^{1-i}\rangle \cong Z_2\times Z_2 & \text{otherwise};
\end{cases}
\end{equation*}

\item[(ii)] if $C$ is as in Example~$\ref{ex:talebi}$ (b), then
\begin{equation*}
A_0= \begin{cases}
\langle \tau_{n/2+1}, \tau_{n/2-1}\vp\rangle \cong D_8 & \text{if $k= n/2-1$ and $n\equiv 0\pmod{4}$,}\\
\langle \tau_{k}\varphi\rangle\cong Z_4 & \text{otherwise}.
\end{cases}
\end{equation*}

\end{enumerate}

\item[(b)] $C$ arises in both cases (a) and (b) of Example~$\ref{ex:talebi}$  if and only if $n\equiv 0\pmod{4}$ and
 $C=\{b, ba, ba^{n/2}, ba^{n/2+1}\}$, which lies in Example~$\ref{ex:talebi}$  (a) with $i=n/2$ or $i=n/2+1$, and in Example~$\ref{ex:talebi}$  (b) with $k=n/2-1$.

 \item[(c)] if $C$ is as in Example~$\ref{ex:talebi}$ (b) then the parameter $k$ is uniquely determined; while if $C$ is as in Example~$\ref{ex:talebi}$ (a), then $C$ arises for exactly two parameters $i, j$, namely $C=\{b, ba, ba^i, ba^{1-i}\}=\{b, ba, ba^j, ba^{1-j}\}$ where $i+j=n+1$.
\end{enumerate}

\end{proposition}

\begin{proof}
(a) First we prove that $|C|=4$. The conditions on the parameters $i, k$ in Example~\ref{ex:talebi} imply that the first three elements of $C$ are pairwise distinct. In case (a) if $ba^{1-i}$ is equal to one of the other elements of $C$ then that element must be $ba^i$ and $i\equiv 1-i\pmod{n}$, so $i=(n+1)/2$. However the conditions on $i$ then imply that $n$ divides $i-1$, which is a contradiction. In case (b), if $ba^{k^2+k+1}$  is equal to one of the other elements of $C$ then $n$ divides $k^2+k+1, k^2+k, k^2$ according as that element is $b, ba, ba^{k+1}$, respectively, and then the fact that $1+k+k^2+k^3\equiv 0\pmod{n}$ implies that $n$ divides $1, k+1, k+1$ respectively, contradicting the fact that $1\leq k\leq n-2$. Thus $|C|=4$.

Now $A_0$ contains the following subgroup $B$ acting regularly on $C$ (which can be checked using \eqref{eq-gen}): for $C$ in case (a), $B=\langle \tau_{n-1}\varphi, \tau_{2i-1}\varphi^{1-i}\rangle\cong Z_2\times Z_2$ (noting that, under the conditions in case (a), $(\tau_{n-1}\vp)^2=(\tau_{2i-1}\vp^{1-i})^2=1$ and the two generators commute); and for $C$ in case (b),   $B=\langle \tau_{k}\varphi\rangle\cong Z_4$ (noting the conditions on $k$).  Thus $C\in\mathcal{T}(G)$.  Also, since $B$ is transitive on $C$, we have a factorisation   $A_0=B(A_0)_b$, and the stabiliser $(A_0)_b$ is contained in $(\Aut(G))_b=T$, as defined in \eqref{eqautd}. The stabiliser of $ba$ in $(A_0)_b$ fixes both $b$ and $a$, and hence is trivial, so $|(A_0)_b|\leq 3$.  Suppose that $(A_0)_b\ne 1$. As this subgroup is contained in $T$, each nontrivial element is equal to $\tau_\ell$, for some $\ell$ with $\gcd(\ell,n)=1$ and $1<\ell<n$, and we have $(ba)^{\tau_\ell}=ba^\ell\in C\setminus\{b,ba\}$.

Suppose first that $C$ is in case (a), so $ba^\ell\in\{ba^i, ba^{1-i}\}$.  Consider first the possibility $ba^\ell=ba^i$. Then $\ell=i$  (since $1<i,\ell<n$) so $\gcd(i,n)=1$ and hence, by the condition on $i$ in Example~!\ref{ex:talebi}, $2(i-1)\equiv 0\pmod{n}$. However $0<2(i-1)<2n$ and it follows that $2(i-1)=n$, so $n$ is even and $\ell=i=n/2+1$. Since $\gcd(\ell,n)=1$ and $n$ is even, this means in particular that $\ell$ is odd and hence $n\equiv 0\pmod{4}$. Thus $\tau_\ell^2=1$ and $A_0=B\langle \tau_\ell\rangle=\langle \tau_{n-1}\varphi, \varphi^{n/2}, \tau_{n/2+1}\rangle = \langle \tau_{n/2+1}, \tau_{n-1}\vp\rangle\cong D_8$, as in (a)(i) (noting that $\tau_{n/2+1} (\tau_{n-1}\vp)=\tau_{n/2-1}\vp$ and $(\tau_{n/2-1}\vp)^2=\vp^{n/2}$). Similarly if $ba^\ell=ba^{1-i}$ then $\ell=n+1-i$  (since $1<i,\ell<n$), and this time $\gcd(i-1,n)=1$, so by the condition on $i$ in Example~!\ref{ex:talebi}, $2i\equiv 0\pmod{n}$, yielding $i=n/2$ and $\ell=n/2+1$. Again since $\gcd(\ell,n)=1$ and $n$ is even,
we must have $\ell$ odd and $n\equiv 0\pmod{4}$. The rest of the argument is the same, and we find that $A_0$ is as in (a)(i). These arguments also show that, if $C$ is as in (a), and if in addition $n\equiv 2\pmod{4}$ whenever $i= n/2$ or $n/2+1$, then
$A_0=B$ as in (a)(i).

Suppose now that $C$ is in case (b), so $ba^\ell\in\{ba^{k+1}, ba^{k^2+k+1}\}$. We claim that $ba^\ell\ne ba^{k+1}$. Suppose to the contrary that $ba^\ell=ba^{k+1}$, so $\ell=k+1$ and $\gcd(k+1,n)=1$. Since $n$ divides $k^3+k^2+k+1 = (k^2+1)(k+1)$ it follows that $k^2+1\equiv 0\pmod{n}$. Thus $ba^{k^2+k+1}=ba^{k}$, and is mapped by $\tau_\ell$ to $ba^{k\ell}=ba^{k^2+k}=ba^{k-1}$. Hence $ba^{k-1} = ba$ or $ba^{k+1}$. If $ba^{k-1} = ba$
then $k=2, \ell=3, n=k^2+1=5$, but then $\tau_\ell: ba \to ba^3\to ba^4\to ba^2\to ba$ does not fix $C$ setwise. Thus   $ba^{k-1} = ba^{k+1}$, but this implies that $a^2=1$ so $n=2$, which is a contradiction, proving the claim. Therefore $ba^\ell= ba^{k^2+k+1}$, and no element of $B$ maps $ba$ to $ba^{k+1}$. This means that $\tau_\ell$ must interchange $ba$ and $ba^{k^2+k+1}$ and fix $ba^{k+1}$. Thus $ba^{k+1}=(ba^{k+1})^{\tau_\ell}=ba^{(k+1)\ell}$ and hence $k+1\equiv (k+1)(k^2+k+1)\equiv k^2+k\pmod{n}$ (using the condition on $k$ in Example~$\ref{ex:talebi}$), so $k^2\equiv 1\pmod{n}$ and $\ell\equiv k^2+k+1\equiv k+2\pmod{n}$. Now the condition in Example~$\ref{ex:talebi}$) becomes $0\equiv k^3+k^2+k+1\equiv 2k+2\pmod{n}$, so $n=2(k+1)/r$ for some integer $r$, and since $k+1<n$ this implies that $r=1$ so $n$ is even and $k=n/2-1$, $\ell=n/2+1$. Further, since $n$ divides $k^2-1$ we must have $n\equiv 0\pmod{4}$, and so $A_0= \langle \tau_{n/2-1}\varphi, \tau_{n/2+1}\rangle\cong D_8$ (note that $(\tau_{n/2-1}\varphi )^2=\varphi^{n/2}$ and the involution $\tau_{n/2+1}$ inverts $\tau_{n/2-1}\varphi$). Thus (a)(ii) holds. In case (b), for all other $k, n$ the group $A_0$ is equal to $B$ and again (a)(ii) holds.

(b) Suppose next that $C$ arises in both cases (a) and (b) of Example~\ref{ex:talebi} for some $i, k$, respectively.
Then $ba^{k+1}$ is equal to either $ba^{i}$ or $ba^{1-i}$. In either case the condition $2i(i-1)\equiv 0\pmod{n}$ implies that $n$ divides $2k(k+1)$. Since $(k+1)(k^2+1)=1+k+k^2+k^3\equiv 0\pmod{n}$, it follows that $n$ divides $(k+1)(2k^2-2(k^2+1))=-2(k+1)$. Thus, since $1\leq k\leq n-2$, we must have $n=2(k+1)$, and so $n$ is even, $k=n/2 - 1$, and $i=n/2$ or $i=n/2+1$ according as $ba^{k+1}=ba^{i}$ or $ba^{1-i}$. The condition in Example~\ref{ex:talebi} (b) on $k$ implies that $\gcd(k,n)=1$, and since $n$ is even this means that $k$ is odd, and hence $n/2$ is even, that is $n\equiv 0\pmod{4}$, and part (b) holds.  The converse is easily checked.

(c) Suppose that $C=\{b, ba, ba^{k+1}, ba^{k^2+k+1}\} =\{b, ba, ba^{\ell+1}, ba^{\ell^2+\ell+1}\}$, where $1\leq k, \ell\leq n-2$, $k\ne \ell$, and  $1+k+k^2+k^3\equiv 1+\ell+\ell^2+\ell^3\equiv 0\pmod{n}$.
Then $ba^{k+1}\ne ba^{\ell+1}$ and so $ba^{k+1} = ba^{\ell^2+\ell+1}$, and hence $k+1\equiv \ell^2+\ell+1\pmod{n}$. Similarly $k^2+k+1\equiv \ell+1\pmod{n}$. Thus, using $1+k+k^2+k^3\equiv 0\pmod{n}$ repeatedly we have
\[
k\equiv \ell(\ell+1)\equiv (k^2+k)(k^2+k+1) \equiv k^3+k^2\equiv -k-1\pmod{n}
\]
so $n$ divides $2k+1$. Since $1<2k+1<2n$ this implies that $n=2k+1$. A similar computation yields $n=2\ell+1$, so $k=\ell$, which is a contradiction. This proves the first assertion in part (c).

Finally suppose that $C=\{b, ba, ba^i, ba^{1-i}\}=\{b, ba, ba^j, ba^{1-j}\}$,  where $2\leq i<j\leq n-1$ such that $\gcd(2i-1, n)=\gcd(2j-1, n)=1$ and $2i(i-1)\equiv 2j(j-1)\equiv 0\pmod{n}$. Then $ba^i\in\{ba^j, ba^{1-j}\}$, and since $ba^i\ne ba^j$ (as $1< i<j<n$), we must have
$ba^i=ba^{1-j}$ which implies that $i=n+1-j$. Conversely if $C=\{b, ba, ba^i, ba^{1-i}\}$ is as in  Example~\ref{ex:talebi}(a) and we set $j=n+1-i$ then $2\leq j\leq n-1$, $\gcd(2j-1, n)=1$ and $2j(j-1)\equiv 0\pmod{n}$. Thus $C$ can be written with either parameter $i$ or $j$. Note that $i, j$ are distinct since $i=j=n+1-i$ implies that $2i-1=n$, which contradicts the condition $\gcd(2i-1, n)=1$.
\end{proof}

\begin{remark}\label{rem:talebi1}
{\rm
(a) It follows from Proposition~\ref{prop:talebi}(c) that, if $C$ is as in Example~$\ref{ex:talebi}$ (a),
then we may assume that the parameter $i$ satisfies $1\leq i\leq n/2$.

(b) We note that there are infinitely many possibilities arising from Example~$\ref{ex:talebi}$ (a), for example, for each even $n\geq 4$ we may take $i=n/2$. On the other hand $n$ cannot be an odd prime power, but some odd composite integers $n$ are possible, for example $n=21, i=7$ gives an example. Similarly there are infinitely many examples arising from Example~$\ref{ex:talebi}$ (b), namely, for each $k\geq2$, we may take $n=k^2+1$, or indeed we may take $n=d(k^2+1)$ for any divisor $d$ of $k+1$.

(c) We note that, by Proposition~\ref{prop:talebi}(a)-(b), taking $n=4$ and $k=1$ in Example~$\ref{ex:talebi}$ (b), or $n=4$ and $i=1$ or $2$ in Example~$\ref{ex:talebi}$ (a), we obtain the graph $\Cay(D_8,C)\cong K_{4,4}$ with $A_0=\langle \tau_3, \varphi\rangle =\Aut(D_8)\cong D_8$.

(d) In his conjecture Talebi writes $p$ in place of $n$ for the subsets we describe in Example~\ref{ex:talebi}(b), and we believe this is simply a misprint. These examples for $n$ prime were identified in \cite[Theorem 3.5]{Talebi}, but they definitely can arise for composite values of $n$. For example, as noted in part (b).
}
\end{remark}

In addition to the two infinite families of examples presented in Example~\ref{ex:talebi}, we have identified a third family of examples and we describe these next.

\begin{example}\label{ex:new}
Let $G=D_{2n}$ and $\Aut(G)$ be as above, with $n\geq5$, and define $C=\{b, ba^i, ba^{j}, ba^{k} \}$, where
\begin{enumerate}
\item[(i)] $1< i, j, k < n$ such that $i, j, k$ are pairwise coprime;

\item[(ii)] $\gcd(x,n)>1$ for each $x\in\{i,j,k\}$; and moreover


\item[(iii)] there exist $\ell, m$ such that $\ell^2\equiv m^2\equiv 1\pmod{n}$ and $\ell, m, \ell m\not\equiv 1\pmod{n}$, and also $k\equiv j\ell+i\equiv im +j\pmod{n}$ and $i(\ell+1)\equiv j(m+1)\equiv 0\pmod{n}$.
\end{enumerate}

\end{example}

\begin{remark}\label{rem:talebi}
{\rm
We note that there
 are additional consequences of the congruence conditions not written explicitly in Example~\ref{ex:new}. For example, since $\ell^2\equiv 1\pmod{n}$, we have $\gcd(\ell,n)=1$, and hence the congruence $k\equiv j\ell+i\pmod{n}$ implies that $\gcd(k-i,n)=\gcd(j\ell,n)=\gcd(j,n)$; similarly $\gcd(j-k,n)=\gcd(i,n)$. Also, using the congruences in part (iii), $j\equiv j\ell^2 \equiv (k-i)\ell \equiv k\ell + i\pmod{n}$, and so $\gcd(i-j,n)=\gcd(-k\ell,n)=\gcd(k,n)$. We present a minimal set of conditions in Example~\ref{ex:new} to make checking as simple as possible. 
}
\end{remark}

\begin{proposition}\label{prop:new}
Let $n, G, C, i, j, k, \ell, m$ be as in Example~$\ref{ex:new}$, and let $A_0=\Aut(G;C)$. Then

\begin{enumerate}
\item[(a)] $|C|=4$, $C\in\mathcal{T}(G)$, and if $n$ is even then $n/2\not\in\{i,j,k\}$;

\item[(b)] $A_0= \langle
\tau_{\ell}\vp^i, \tau_m\vp^j\rangle\cong Z_2\times Z_2$.
\end{enumerate}
Moreover, $C^\sigma$ does not occur in
Example~$\ref{ex:talebi}$, for any $\sigma\in \Aut(G)$.
\end{proposition}

\begin{proof}
By the conditions on $i,j,k$ it follows that $C$ consists of four pairwise distinct involutions, and $C$ generates $G$. Next we show that, if $n$ is even then $n/2\not\in\{i,j,k\}$. Suppose to the contrary that, say, $n=2i$. Then, using the condition that  $i, j$ are coprime, $\gcd(n,j)=\gcd(2i,j)=\gcd(2,j)$, and hence since $\gcd(j,n)>1$ it follows that $j$ is even. A similar proof shows that $k$ is even, and this is a contradiction since $j, k$ are coprime.

Now we check that both $\tau_{\ell}\vp^i$ and $\tau_m\vp^j$ leave $C$ invariant.  Let $\sigma= \tau_{\ell}\vp^i$. Then, using \eqref{eq-gen}, $b^\sigma = ba^i$, $(ba^i)^\sigma=ba^{i\ell+i}=b$ (since $i(\ell+1)\equiv 0\pmod{n}$),  $(ba^j)^\sigma=ba^{j\ell+i}=ba^k$ (since $j\ell+i\equiv k\pmod{n}$), and
 $(ba^k)^\sigma=ba^{k\ell+i}=ba^j$, since
 $$
 k\ell+i\equiv (j\ell+i)\ell+i =
 j\ell^2+i(\ell+1)\equiv j.1+0=j\pmod{n}.
 $$
 Hence $\sigma= \tau_{\ell}\vp^i\in A_0$, and also $|\sigma|=2$ (since $A_0$ is faithful on $C$). Now let $\sigma= \tau_{m}\vp^j$. Then $b^\sigma = ba^j$, $(ba^j)^\sigma=ba^{jm+j}=b$ (since $j(m+1)\equiv 0\pmod{n}$),  $(ba^i)^\sigma=ba^{im+j}=ba^k$ (since $im+j\equiv k\pmod{n}$), and
 $(ba^k)^\sigma=ba^{km+j}=ba^i$, since
 $$
 km+j\equiv (im+j)m+j =
 im^2+j(m+1)\equiv i.1+0=i\pmod{n}.
 $$
 Hence $\sigma= \tau_{m}\vp^j\in A_0$, and  again, since $A_0$ is faithful on $C$, we conclude that
 $|\tau_{m}\vp^j|=2$ and $A_0$ contains $A:=\langle\tau_{\ell}\vp^i, \tau_m\vp^j\rangle$. The computations above show that the generators act on $C$ as double transpositions and that, in this action, $A$ induces the Klein four group on $C$. Since $A$ is faithful on $C$ it follows that  $A\cong Z_2\times Z_2$, and $A$ is transitive on $C$, so $C\in\mathcal{T}(G)$, proving part (a).

We claim that $A_0\cap \langle \varphi\rangle = 1$. Suppose to the contrary that  $A_0$ contains $\varphi^u$ with $0<u<n$. By \eqref{eqvpt}, for each $x$, $(ba^x)^{\varphi^u}=ba^{u+x}\ne ba^x$, and hence the permutation induced by $\varphi^u$ on $C$ is either a product of two $2$-cycles, or a $4$-cycle. 
In the former case, since  $\varphi^u: b\to ba^u \to ba^{2u}$, we have $b=ba^{2u}$ and hence $n$ divides $2u$. This implies that $n$ is even, $u=n/2$, and $ba^{n/2}\in C$ so $n/2\in\{i,j,k\}$, which  contradicts part (a). Hence $\varphi^u$ induces a $4$-cycle on $C$, and since $\varphi^u: b\to ba^u \to ba^{2u} \to ba^{3u} \to ba^{4u}$, it follows that $n$ divides $4u$. Since the images $b, ba^u, ba^{2u}, ba^{3u}$ are pairwise distinct, $n$ does not divide $2u$, so $n=4y$ where $y$ divides $u$.  Without loss of generality $i=u$ and $j\equiv 2u\pmod{n}$. Thus $y$ divides $\gcd(i,j)$, and since $i,j$ are coprime this implies that $y=1$. Hence $n=4$, which is a contradiction. Thus $A_0\cap \langle\varphi\rangle=1$.
Therefore $A_0\cong
A_0\langle\vp\rangle/\langle\vp\rangle\leq \Aut(G)/\langle\vp\rangle\cong T$, as in \eqref{eqautd}. In particular $A_0$ is abelian.   Thus, since $A_0$ acts faithfully on $C$ and contains $A$, we conclude that $A_0=A$, proving part (b).

 Finally we prove the last assertion. Suppose to the contrary that, for some $\sigma\in\Aut(G)$, the image $C^\sigma$ is as in Example~\ref{ex:talebi}. In particular,  $C^\sigma$ contains $b$ and $ba$. Since $A_0 = Z_2\times Z_2$, it follows from Proposition~\ref{prop:talebi} that $C^\sigma$ is in Example~\ref{ex:talebi}(a). Suppose first that $\sigma=\varphi^u$ with $0\leq u<n$. Then $C^\sigma$ consists of the four elements $(ba^x)^\sigma = ba^{x+u}$ for $x\in\{0,i,j,k\}$. If $u=0$ then $\sigma$ is the identity, and this is not possible since $ba\not\in C$. Hence $u\ne 0$, so without loss of generality, $b=ba^{i+u}$, that is $u
\equiv -i\pmod{n}$. Then for some $x\in\{0, j, k\}$, $ba=ba^{x+u}=ba^{x-i}$. However this implies that $x-i\equiv 1\pmod{n}$; and, using Remark~\ref{rem:talebi} and condition (ii) of Example~\ref{ex:new},  $\gcd(i,n)>1$, $\gcd(j-i,n)=\gcd(k,n)>1$, and $\gcd(k-i.n)=\gcd(j,n)>1$, and we have a contradiction.  Hence $\sigma=\tau_y\varphi^u$ for some $y, u$ with $\gcd(y,n)=1$ and  $y\not\equiv 1\pmod{n}$.
If $u=0$ then $b=b^\sigma$, and without loss of generality $ba=(ba^i)^\sigma = ba^{iy}$, implying that $iy\equiv 1\pmod{n}$, which is not possible since $\gcd(i,n)>1$. Hence $u>0$ and without loss of generality we may assume that $b=(ba^i)^\sigma = ba^{iy+u}$, so $u\equiv -iy\pmod{n}$. Thus, for some $x\in\{0,j,k\}$, $ba=(ba^x)^\sigma = ba^{xy+u} = ba^{(x-i)y}$, which implies that $\gcd(x-i, n)=1$, and the same argument which we just gave leads to a contradiction.
\end{proof}

In Subsection~\ref{sub-bk} we exhibit an explicit infinite family of possibilities for $C$ which satisfy all the conditions of Example~\ref{ex:new}. Thus this new construction disproves Talebi's conjecture.
We now use the framework from Theorem~\ref{lem:d2n} to prove Theorem~\ref{4valent}, which classifies all $4$-element subsets of
$\mathcal{T}(D_{2n})$, and hence all $4$-valent  normal edge-transitive Cayley graphs for dihedral groups.

\begin{proof} \emph{ of Theorem \ref{4valent}.~}\quad
Let $G=D_{2n}$ with $n\geq3$, as in $\eqref{eqd}$, and $\Aut(G)$ as in $\eqref{eqautd}$. If $C$ is as in
Example~$\ref{ex:talebi}$ or Example~$\ref{ex:new}$, then  $\Cay(G;C)$ is a $4$-valent normal edge-transitive Cayley graph by Proposition~\ref{prop:talebi} or Proposition~\ref{prop:new}, respectively. The same is true for the images of such subsets $C$ under elements of $\Aut(G)$, by Lemma~\ref{lem:tg}(a).

Now we prove the converse. So suppose that $C$ is an inverse-closed generating set for $G$ and  $\Cay(G;C)$ is a $4$-valent normal edge-transitive Cayley graph. Then $|C|=4$ and, by \cite[Proposition 1(b)]{praeger},  $C\in\mathcal{T}(G)$. Let $A_0=\Aut(G;C)$, and $A_0\cap \langle\vp\rangle = \la \vp^r\rangle$ where $r$ divides $n$.
By Theorem~\ref{lem:d2n}, $C=\cup_{i\in I} ba^i\langle a^r\rangle$ for some $I\subseteq \{0,1,\dots,r-1\}$ and we may assume that $0\in I$. Moreover, $A_0/\langle\vp^r\rangle$ acts transitively on the set
$\{ba^i\langle a^r\rangle | i\in I\}$. In particular $4=|C|=|I|n/r\geq |I|$, so one of the following holds: (i) $I=\{0\}$ and $r=n/4$, (ii) $|I|=2$ and $r=n/2$, or (iii) $|I|=4$ and $r=n$.

\medskip\noindent
\emph{Case (i)}\ Here, by Theorem~\ref{lem:d2n}(c), $r=1$, so $n=4r=4$,   $C=b\la a\ra$, $A_0=\Aut(G)$, and $\Cay(G;C)\cong K_{4,4}$, see Example~\ref{ex:frat3}(d). By Remark~\ref{rem:talebi1}, $C$ is as in Example~\ref{ex:talebi}.

\medskip\noindent
\emph{Case (ii)}\ In this case $I=\{0,i\}$ for some $i$, and $C = b\{1,a^i,a^{n/2}, a^{n/2+i}\}$. By Theorem~\ref{lem:d2n}(d), $1=\gcd(i,r)=\gcd(i,n/2)$, and there exists a positive integer $k<n$ such that $\gcd(k,n)=1$, and $r=n/2$ divides $k+1$. This implies that $k=n/2-1$ or $k=n-1$. Moreover. by Theorem~\ref{lem:d2n}(d), $B:=\langle \vp^{n/2}, \tau_{k}\vp^i\rangle\leq A_0$ and $B$ acts transitively on $C$.

Suppose first that $k=n/2-1$, so $n$ is even. Then, since $\gcd(k,n)=1$, it follows that $k$ is odd, and hence $n\equiv 0\pmod{4}$ (since $k=n/2-1$). This implies that $i$ is odd (since $\gcd(i,n/2)=1$), and hence $\gcd(i,n)=1$. Thus there exists a positive integer $\ell<n$ such that $i\ell\equiv 1\pmod{n}$, so $\gcd(\ell,n)=1$ and applying the map $\tau_\ell\in\Aut(G)$ defined in \eqref{eqvpt} to $C$ we have $C^{\tau_\ell} = b\{1,a,a^{n/2}, a^{n/2+1}\}$, as in Example~\ref{ex:talebi} (see Proposition~\ref{prop:talebi}(b)).

Suppose now that $k=n-1$. Again $n$ is even since $r=n/2$. Suppose first that $\gcd(i,n)=1$. Then, as in the previous paragraph, there exists a positive integer $\ell<n$ such that  $i\ell\equiv 1\pmod{n}$ and $\gcd(\ell,n)=1$, and hence $C^{\tau_\ell} = b\{1,a,a^{n/2}, a^{n/2+1}\}$, as in Example~\ref{ex:talebi}. Suppose on the other hand that $\gcd(i,n)>1$. Since $\gcd(i,n/2)=1$, this implies that $i$ is even, $n/2$ is odd, and $\gcd(i,n)=2$. Thus there exists   a positive integer $\ell'<n/2$ such that $i\ell'\equiv 1\pmod{n/2}$, whence $\gcd(\ell',n/2)=1$. If $\ell'$ is odd set $\ell=\ell'$, and if $\ell'$ is even set $\ell=\ell'+n/2$. Then in either case, $\ell$ is odd, $1\leq \ell<n$, and $\gcd(\ell,n)=1$.
Thus $\tau_\ell\in\Aut(G)$. Now we have $i\ell \equiv 1\pmod{n/2}$ with $n/2, \ell$ odd and $n, i$ even; it follows that $i\ell\equiv 1+n/2\pmod{n}$. Therefore $(a^i)^{\tau_\ell}=a^{i\ell}=a^{n/2+1},
(a^{n/2})^{\tau_\ell}=a^{(n/2)\ell}=a^{n/2}$, and $(a^{n/2+i})^{\tau_\ell}=a^{(n/2+i)\ell}=a^{n/2+n/2+1}=a$. Hence $C^{\tau_\ell}=b\{1, a, a^{n/2}, a^{n/2+1}\}$, as in Example~\ref{ex:talebi}.

\medskip\noindent
\emph{Case (iii)}\ In this case $|I|=4$ and $r=n$, and we have $C=b\{1, a^{i_1}, a^{i_2}, a^{i_3} \}$, for distinct $i_1, i_2, i_3$ between $1$ and $n-1$. Also $A_0\cap\langle\vp\rangle=1$ so $A_0\cong
A_0\langle\vp\rangle/\langle\vp\rangle\leq \Aut(G)/\langle\vp\rangle\cong T$. In particular $A_0$ is abelian and each nontrivial element of $A_0$ is of the form $\tau_k\vp^i$ for some $k, i$ with $\gcd(k,n)=1$ and $0\leq i<n$. Since $C$ generates $G$ it follows that $A_0$ acts faithfully on $C$,
so $A_0$ is a transitive abelian permutation group on $C$. This implies that $|A_0|=4$ and either $A_0$ is cyclic or a Klein four group.

Suppose first that $A_0$ is cyclic, say $A_0=\langle \tau_k\vp^i\rangle\cong Z_4$. Without loss of generality we may assume that $\tau_k\vp^i:b\to ba^{i_1} \to ba^{i_2} \to ba^{i_3} \to b$. By \eqref{eq-gen}, $i_1=i$, $i_2\equiv i(k+1)\pmod{n}$, $i_3\equiv i(k^2+k+1)\pmod{n}$, and $i(k^3+k^2+k+1)\equiv 0\pmod{n}$. This implies in particular that $G=\langle C\rangle\leq \langle b, a^i\rangle$, and we conclude that $\gcd(i,n)=1$. Thus there exists $\ell$ such that $\gcd(\ell,n)=1$ and $i\ell\equiv 1\pmod{n}$, and replacing $C$ by $C^{\tau_\ell}$ we may assume in addition that $i=1$.  Hence $C$ is as in Example~\ref{ex:talebi}(b).

Suppose now that $A_0\cong Z_2\times Z_2$. Then $A_0$ has generators $\tau_\ell\vp^i$ and $\tau_{m}\vp^j$ such that $\tau_\ell\vp^i: b\leftrightarrow ba^{i_1},  ba^{i_2}\leftrightarrow  ba^{i_3}$ and
$\tau_{m}\vp^j : b\leftrightarrow  ba^{i_2},  ba^{i_1}\leftrightarrow  ba^{i_3}$. This implies that $i_1=i, i_2=j$, $i_3\equiv j\ell +i\equiv  im+j\pmod{n}$, and $i(\ell+1)\equiv j(m+1)\equiv 0\pmod{n}$. Thus $a^{i_3}=a^{j\ell+i}\in\langle a^i, a^j\rangle$, and since $C$ generates $G$ it follows that
$G=\langle b, a^i, a^j\rangle$, and hence $\gcd(i,j)=1$. A similar argument shows that $\gcd(i,i_3)=\gcd(j,i_3)=1$. We will write $k=i_3$, and we note that the third involution in $A_0$ is $(\tau_\ell\varphi^i)(\tau_m\varphi^j)=\tau_{\ell m}\varphi^{im+j}=\tau_{\ell m}\varphi^{k}$.
Now $(\tau_\ell\varphi^i)^2=1$, and \eqref{eqvpt} implies that $\ell^2\equiv 1\pmod{n}$. Similarly
$m^2\equiv (\ell m)^2\equiv 1\pmod{n}$. Also since $A_0\cap\langle\varphi\rangle = 1$ (as we are in case
(iii)) we have $\ell, m, \ell m\not\equiv 1\pmod{n}$. This means that all the conditions of
Example~\ref{ex:new} parts (i) and (iii) hold. Using the congruences we have just derived, we have, modulo $n$,
\[
k(\ell m+1)\equiv (im+j)(\ell m+1) = i\ell m^2 + im + j\ell m +j \equiv -im^2+im -j\ell + j\equiv -im^2 +i = -i(m^2-1)\equiv 0.
\]
Thus the conditions on the pairs $(i,k), (j,\ell), (k,\ell m)$ are `symmetrical'.

We claim that, \emph{if any one of $\gcd(i,n)=1$, $\gcd(j,n)=1$, or $\gcd(k,n)=1$, 
then $C^\sigma$ is as in   Example~\ref{ex:talebi}(a), for some $\sigma\in\Aut(G)$.}
If $\gcd(i,n)=1$, then there exists $u$ such that $iu\equiv 1\pmod{n}$, and the image of $C$ under $\tau_u$ is $\{b, ba, ba^{ju}, ba^{i_3u}\}$. Similar arguments show that, if  $\gcd(j,n)=1$, or $\gcd(k,n)=1$,  then the image of $C$ under some $\sigma\in\Aut(G)$ contains both $b$ and $ba$. 
Thus (because of the symmetry of the conditions on $i, j, k$) in order to prove the claim we may assume that $i=1$ and hence $2\leq j,k\leq n-1$. Then, since
$\ell+1=i(\ell+1)\equiv 0\pmod{n}$, we have $\ell=n-1$, and hence $k\equiv j\ell+i \equiv -j+1\pmod{n}$. Since also $k\equiv im+j=m+j\pmod{n}$, we have $m\equiv k-j\equiv 1-2j\pmod{n}$, so $\gcd(2j-1,n)=1$ and $0\equiv j(m+1)\equiv 2j(1-j)\pmod{n}$. Thus $C=\{b,ba,ba^j, ba^{1-j}\}$ is as in Example~\ref{ex:talebi}(a), proving the claim.

To complete the proof of Theorem~\ref{4valent}, we may therefore assume also that all the conditions of Example~\ref{ex:new} part (ii) hold, so $C$ is as in Example~\ref{ex:new}, completing the proof of Theorem \ref{4valent}.
\end{proof}

\subsection{Infinite family of Examples}\label{sub-bk}
 By Remark~\ref{rem:talebi1}(b), there are infinitely many $4$-valent normal edge-transitive
 Cayley graphs of dihedral groups arising in each part of Example \ref{ex:talebi}. Here we present an explicit infinite family of examples for our new construction in Example~\ref{ex:new}. Proof that we obtain infinitely many examples simply depends on the fact that there are infinitely many odd primes $p$ not of the form $p=2^r-1$, that is, primes which are not \emph{Mersenne primes}.

\begin{lemma}\label{mersenne}
There are infinitely many odd primes which are not Mersenne primes.
\end{lemma}

\begin{proof}
This follows easily from the number theoretic result known as Bertrand's Postulate (see for example \cite[Theorem 8.7]{NZM}) that, for all integers $m\geq3$ there exists a prime $p$ such that $m<p<2m$.
Taking $m=2^r-1\geq3$, we obtain a prime satisfying $2^r-1<p<2^{r+1}-2$. Such a prime $p$ is odd and is not a Mersenne prime, and the primes obtained for distinct integers $r\geq2$ are pairwise distinct.
\end{proof}

\begin{proposition}\label{prop:bk}
Let $p$ be an odd prime that is not a Mersenne prime, and let $q$ be any odd prime divisor of $p+1$.
Define $n, i,j,k,\ell,m$ as follows:
\[
n=2pq,\ i=p,\ j=q,\ k=p+q,\ \ell = 2p+1,\ m= n-\ell =2pq-2p-1.
\]
Then all the conditions of Example~$\ref{ex:new}$ hold for $i,j,k,\ell,m$, and in particular $C=\{b,ba^i,ba^j,ba^k\}$ arises in Example~$\ref{ex:new}$ for the group $G=D_{2n}$.
\end{proposition}

\begin{proof}
Since $p, q$ are distinct odd primes it follows that $i,j,k$ are pairwise coprime, so condition (i) of  Example~$\ref{ex:new}$ holds. Next
\[
\gcd(i,n)=p,\quad \gcd(j,n)=q,\quad \gcd(k,n)=2,
\]
and hence  condition (ii) of  Example~$\ref{ex:new}$ holds. Now we check condition (iii)
of Example~$\ref{ex:new}$. Note that $p$ divides $\ell-1$ and $q$ divides  $2(p+1) = \ell+1$,
and hence, since $\ell$ is odd, $n=2pq$ divides $\ell^2-1$, that is, $\ell^2\equiv 1\pmod{n}$.
Also $m^2\equiv 1\pmod{n}$ since $m\equiv -\ell\pmod{n}$.  By definition, $\ell, m\not\equiv
1\pmod{n}$, and we have $\ell m \equiv -\ell^2 \equiv -1\not\equiv 1\pmod{n}$. Next $i(\ell+1)
=p(2p+2)=2p(p+1)\equiv 0\pmod{n}$ since $q$ divides $p+1$, and $j(m+1)= q(n-2p)\equiv -2pq
\equiv 0\pmod{n}$. Also $j\ell+i=q(2p+1)+p\equiv q+p=k\pmod{n}$. Finally
$im+j=p(n-2p-1)+q\equiv -2p(p+1) +p+q \equiv p+q= k\pmod{n}$, since $n$ divides $2p(p+1)$.
Thus all the conditions of Example~$\ref{ex:new}$ hold, so $C=\{b,ba^i,ba^j,ba^k\}$ arises in Example~$\ref{ex:new}$ for the group $G=D_{2n}$.
\end{proof}

\section{Graphs which are not normal edge-transitive}

For an arbitrary inverse-closed generating set $C$ of a group $G$, it is of course possible to define normal quotients of the Cayley graph $\Cay(G;C)$ relative to $\Aut(G;C)$-invariant normal subgroups of $G$. One may wonder to what extent Theorem~\ref{main decomposition} might extend to this general situation, that is to say, whether or not the assumption that $\Cay(G;C)$ is normal edge-transitive is needed. Here we consider a small example which pin-points one of the problems which may cause the graph isomorphism in  Theorem~\ref{main decomposition}(ii) to fail if  $\Cay(G;C)$ is not normal edge-transitive. Recall that $\caycs$ is the class of connected normal edge-transitive Cayley graphs for finite characteristically simple groups.

\begin{example}\label{the converse is not true}
		Let $G=\mathbb{Z}_2\times \mathbb{Z}_2\times \mathbb{Z}_2\times \mathbb{Z}_2$ so that $\Aut(G)=\GL(4,2)$, and define $C=\{a_1, a_2, b_1, b_2, b_3\}$, where 
\[		
		a_1=(1,0,1,1),\ a_2=(0,1,1,1),\ b_1=(1,0,1,0),\ b_2=(0,1,0,1),\ b_3=(1,1,1,1),
\]
and set $A=\{a_1,a_2\}$ and $B=\{b_1, b_2, b_3\}$, so 		 $C=A\cup B$. Define
\begin{align*}
\tau_1 &: (x,y,z,w)\to (y+z+w, x+z+w, z, w),\\ 
\tau_2 &: (x,y,z,w)\to (x+z+w, x+w, w, x+y),\ \text{and}\\
\tau_3 &:(x,y,z,w)\to (y,x,w,z).  
\end{align*}
\end{example} 

\begin{proposition}\label{lem:converse1}
Let $G, C$ be as in Example~$\ref{the converse is not true}$, and let $A_0=\Aut(G;C)$. Then 
\begin{enumerate}
\item[(a)] $C$ is an inverse-closed generating set for $G$, and $A_0=\langle \tau_1\rangle\times \langle \tau_2, \tau_1\tau_3\rangle \cong \Sym(A)\times\Sym(B)= S_2\times S_3$. In particular $A_0$ fixes each of $A$ and $B$ setwise, and $C\not\in\mathcal{T}(G)$, that is, $C$ is not transitive.

\item[(b)] $\calA_{\max}(G;C)$ contains $M:=\langle A\rangle\cong \mathbb{Z}_2^2$ and $N:=\langle b_1, b_2, a_1+a_2\rangle\cong\mathbb{Z}_2^3$, and  $\Phi(G;C)=M\cap N = \langle a_1+a_2\rangle\cong \mathbb{Z}_2$.

\item[(c)] The graph $\Gamma :=\Cay(G;C)$ is not normal edge-transitive; its normal quotient graphs  
\[
\Gamma_M:=\Cay(G/M,(C+M)/M),\ \Gamma_N:=\Cay(G/N,(C+N)/N)\in \caycs,
\] 
and moreover $\Gamma_M\cong K_4$, $\Gamma_N\cong K_2$, while for $\Phi:=\Phi(G;C)$,  $\Gamma_\Phi:=\Cay(G/\Phi,(C+\Phi)/\Phi)$ is the $4$-valent complement of the cube $Q_3$.

\item[(d)]  The map $g\to (gM, gN)$ induces a group isomorphism $G/\Phi(G;C)\to G/M\times G/N$, but does not induce a graph isomorphism from $\Gamma_\Phi$ to a subgraph of the direct product $\Gamma_M\times \Gamma_N$. For example, a `$b_3$-type' edge such as $\{ 0+\Phi, b_3+\Phi\}$ of $\Gamma_\Phi$ does not map to an edge of $\Gamma_M\times \Gamma_N$ as it projects to a loop of $\Gamma_N$.
\end{enumerate}
\end{proposition}
	 
\begin{proof}
(a) The first assertion follows since each element of $G$ is equal to its inverse, and since $C$ spans $G$ as a vector space. Next, to simplify our later arguments, we check directly from their definitions that 
\begin{align*}
\tau_1 &: a_1\leftrightarrow a_2,\ b_1\to b_1,\   b_2\to  b_2,\   b_3\to  b_3\\ 
\tau_2 &:  a_1\to a_1,\  a_2\to a_2,\ b_1\to b_2\to b_3\to b_1\\
\tau_3 &: a_1\leftrightarrow a_2,\ b_1\leftrightarrow b_2,\ b_3\to b_3.   
\end{align*}
In particular each $\tau_i$ maps the spanning set $C$ bijectively to itself and hence each $\tau_i\in\GL(4,2)$. Again, since each $\tau_i$ fixes $C$ setwise (and moreover each $\tau_i$ fixes each of $A$ and $B$ setwise), it follows that each $\tau_i\in A_0$. Now $A_0$ acts faithfully as a permutation group on $C$, and hence it follows easily from the displayed actions of the $\tau_i$ on $C$ that the subgroup they generate is $H =\langle\tau_1\rangle\times \langle \tau_2, \tau_1\tau_3\rangle \cong \Sym(A)\times\Sym(B)= S_2\times S_3$. 
Moreover, this subgroup $H$ is a maximal subgroup of  $\Sym(C)=S_5$. If $A_0$ is strictly larger than $H$, then we would have $A_0\cong \Sym(C)$, and hence $A_0$ would contain an element $\varphi: b_1\to b_1,\   b_2\to  b_2,\   b_3\to a_1\to a_2\to b_3$. However such an element $\varphi$ must in particular lie in $\GL(4,2)$ and hence 
\[
\varphi(b_3)=\varphi(b_1+b_2)=\varphi(b_1)+\varphi(b_2)=b_1+b_2=b_3,
\]
which is a contradiction. Hence $H=A_0$. Thus $A_0$ has orbits in $C$ of lengths $2$ and $3$, and hence  $C\not\in\mathcal{T}(G)$, that is, $C$ is not transitive. This proves part (a).

(b) Let $M_B=\langle B\rangle = \{0,b_1, b_2, b_3\}$. Then $G=M+M_B$ and $M\cap M_B=0$. In particular each nontrivial coset of $M$ in $G$ is of the form $b_i+M$ for some $i\leq 3$. Thus if $K$ is an $A_0$-invariant normal subgroup of $G$ properly containing $M$ then $K=G$ since $A_0$ is transitive on $B$ by part (a). It follows that $M\in \calA_{\max}(G;C)$.  Next we deal with $N$. We see by its definition that $N=\{(x+z, y+z, x,y)\mid x,y,z\in\mathbf{F}_2\}$. It follows from the actions of the $\tau_i$ displayed above that $\tau_1$ fixes $N$ pointwise, and that both $\tau_2$ and $\tau_3$ leave $N$ invariant. Hence $N$ is an $A_0$-invariant normal subgroup of $G$, and as $N$ is maximal in $G$ we conclude that $N\in \calA_{\max}(G;C)$. We note that $M\cap N=\langle a_1+a_2\rangle = \{0, (1,1,0,0)\}$ since $M\cap M_B=0$.  Hence, $\Phi(G;C)\leq 
M\cap N$, by \eqref{eq:phi}, so to complete the proof of part (b) it remains for us to show that $a:=a_1+a_2=(1,1,0,0)\in\Phi(G;C)$. 

To do this, we need to check that the criterion of Theorem~\ref{thm:frat}(b) holds for $a$. So suppose that $X\subseteq G$ and that the normal $C$-closure of $X\cup \{a\}$ generates $G$, that is, $G$ is generated by 
$\widehat{X}:=\{ (x^g)^\sigma | x\in X\cup\{a\}, g\in G, \sigma\in A_0\}$. Since $G$ is abelian, and since $a=a_1+a_2$ is fixed by each element of $A_0$ (by part (a)), $\widehat{X}=\{a\}\cup Y$, where $Y=\{(x^g)^\sigma|x\in X, g\in G, \sigma\in A_0\}$, that is, $Y$ is the normal $C$-closure of $X$. We therefore need to prove that $Y$ generates $G$. Suppose to the contrary that $W:=\langle Y\rangle$ is a proper subgroup of $G$. Since $Y\cup\{a\}$ generates $G$ it follows that $G=W + \langle a\rangle$ and $W\cap \langle a\rangle = 0$. In particular there exists $w\in W$ such that $w+a=a_1$. This means that $W$ contains $w=a+a_1=a_2$, and since $Y$ is $C$-closed and generates $W$, also $a_1=a_2^{\tau_1}$ lies in $W$, and therefore $a=a_1+a_2\in W$, which is a contradiction. Thus we have proved that $Y$ generates $G$, and hence that the criterion of  Theorem~\ref{thm:frat}(b) holds for $a$.  Therefore $a:=(1,1,0,0)\in\Phi(G;C)$.  and part (b) is proved.

(c) By \cite[Proposition 1(c)]{praeger}, the graph $\Gamma :=\Cay(G;C)$ is not normal edge-transitive since the set $C$ is not transitive (by part (a)). Since $(C+M)/M=(B+M)/M$ (as $A\subset M$), and since $A_0$ is transitive on $B$ by part (a), it follows that $\Gamma_M$ is normal edge-transitive (by \cite[Proposition 1(c)]{praeger}), and hence $\Gamma_M\in\caycs$. Similarly $\Gamma_N\in\caycs$ since $(C+N)/N=(A+N)/N$. Now $\Gamma_M\cong K_4$ and $\Gamma_N\cong K_2$, since $(C+M)/M$ and $(C+N)/N$ consist of all the nontrivial cosets of $G/M$ and $G/N$, respectively. On the other hand, $(C+\Phi)/\Phi$ has size 4, since $a_1+\Phi=a_2+\Phi$ and this element is distinct from each of the $b_i+\Phi$, which themselves are pairwise distinct. Thus $\Gamma_\Phi$ has valency $4$, and it is not difficult to check that $\Gamma_\Phi$  is the complement of the $3$-dimensional cube graph $Q_3$.  

(d) The map $\phi:g\to (gM, gN)$ is a group homomorphism $\phi:G\to G/M\times G/N$ with kernel $M\cap N=\Phi(G;C)$ (by part (b)). Thus the image of $\phi$ is isomorphic to $G/\Phi(G;C)=G/(M\cap N)\cong \mathbb{Z}_2^3$. Since $G/M\cong \mathbb{Z}_2^2$ and $G/N\cong\mathbb{Z}_2$, it follows that $\phi$ is onto and $\phi$ induces a group isomorphism $G/\Phi(G;C)\cong G/M\times G/N$. This bijection on vertices does not map all edges of $\Gamma_\Phi$ to edges of the direct product $\Gamma_M\times\Gamma_N$, as a `$b_3$-type'  edge $\{x+\Phi, x+b_3+\Phi\}$ is mapped to the vertex-pair $\{(x+M, x+N),  (x+b_3+M, x+N)\}$ (since $b_3\in N$), and this is not an edge of the direct product $\Gamma_M\times\Gamma_N$. Put more simply, 
$\Gamma_\Phi$ has valency $4$ and so is not isomorphic to a subgraph of the valency $3$ direct product graph $\Gamma_M\times\Gamma_N$.
\end{proof}

\end{document}